\documentclass[12pt,reqno]{amsart}
\usepackage{amssymb}
\usepackage{enumerate}
\usepackage{graphicx}
\usepackage{hyperref}

\usepackage{bm}

\usepackage[T1]{fontenc}
\usepackage[all]{xy}
\usepackage{tikz}
\usetikzlibrary{arrows}

\definecolor{darkgreen}{RGB}{65,165,118}
\definecolor{darkred}{RGB}{180,0,0}

\let\xto=\xrightarrow

\newcommand{\higherlim}[2]{\displaystyle\setbox1=\hbox{\rm lim}
	\setbox2=\hbox to \wd1{\leftarrowfill} \ht2=0pt \dp2=-1pt
	\setbox3=\hbox{$\scriptstyle{#1}$}
	\def\test{#1}\ifx\test\empty
	\mathop{\mathop{\vtop{\baselineskip=5pt\box1\box2}}}\nolimits^{#2}
	\else
	\ifdim\wd1<\wd3
	\mathop{\hphantom{^{#2}}\vtop{\baselineskip=5pt\box1\box2}^{#2}}_{#1}
	\else
	\mathop{\mathop{\vtop{\baselineskip=5pt\box1\box2}}_{#1}}%
	\nolimits^{#2}
	\fi\fi}

\let\oldcirc=\circ
\renewcommand{\circ}{\mathchoice
    {\mathbin{\scriptstyle\oldcirc}}{\mathbin{\scriptstyle\oldcirc}}
    {\mathbin{\scriptscriptstyle\oldcirc}}
    {\mathbin{\scriptscriptstyle\oldcirc}}}

\SelectTips{cm}{10} \UseTips   

\numberwithin{equation}{section}

\mathcode`\:="003A
\mathchardef\cdot="0201

\renewenvironment{enumerate}[1][]
{\begin{enumerat}[#1]\setlength{\itemsep}{6pt}}{\end{enumerat}}

\renewenvironment{itemize}
{\begin{itemiz}\setlength{\itemsep}{6pt}\setlength{\itemindent}{-20pt}}
{\end{itemiz}}

\def\beq#1\eeq{\begin{equation*}#1\end{equation*}}
\def\beqq#1\eeqq{\begin{equation}#1\end{equation}}

\let\emptyset=\varnothing

\DeclareMathAlphabet\EuR{U}{eur}{m}{n}
\SetMathAlphabet\EuR{bold}{U}{eur}{b}{n}
\newcommand{\curs}{\EuR}
\newcommand{\Ab}{\curs{Ab}}

\renewcommand{\:}{\colon}

\newlength{\upto}\newlength{\dnto}
\newcounter{let} 
\loop\stepcounter{let}
\expandafter\edef\csname cal\alph{let}\endcsname%
{\noexpand\mathcal{\Alph{let}}}
\ifnum\thelet<26\repeat

\newcommand{\tdef}[2][]{\expandafter\newcommand\csname#2\endcsname%
{#1\textup{#2}}}
\tdef{Iso}   \tdef{Aut}    \tdef{Out}    \tdef{Inn}    \tdef{Hom}
\tdef{End}   \tdef{Inj}    \tdef{map}    \tdef{Ker}    \tdef{Ob}
\tdef{Mor}   \tdef{Res}    \tdef{Spin}   \tdef{Id}     \tdef{Fr}
\tdef{rk}    \tdef{conj}   \tdef{incl}   \tdef{proj}   \tdef{diag}
\tdef{trf}   \tdef{Sol}    \tdef{He}     \tdef{Sz}     \tdef{cj}
\tdef{free}  \tdef{ev}     \tdef{Map}    \tdef{Rep}    \tdef{pr}
\tdef{supp}  \tdef{res}    \tdef{hofibre}  \tdef{Vect}  \tdef{hofib}
\tdef{Coker}
\tdef[_]{fc}  \tdef[_]{fn}  \tdef[_]{typ} \tdef[^]{op}   \tdef[^]{ab}

\newcommand{\colim}{\mathop{\textup{colim}}\limits}

\newcommand{\fdef}[1]{\expandafter\newcommand\csname#1\endcsname%
{\mathfrak{#1}}}
\fdef{X}  \fdef{red}  \fdef{foc}  \fdef{hyp}  \fdef{Sub}

\newcommand{\bbdef}[1]{\expandafter\newcommand%
\csname#1\endcsname{\mathbb{#1}}}
\bbdef{C} \bbdef{F} \bbdef{R} \bbdef{Z} \bbdef{Q}

\newcommand{\defeq}{\overset{\textup{def}}{=}}

\newcommand{\aaa}{\mathbf{a}}
\newcommand{\sss}{\mathbf{s}}
\newcommand{\eee}{\mathbf{e}}

\newtheorem{Thm}{Theorem}[section]
\newtheorem{Prop}[Thm]{Proposition}
\newtheorem{Cor}[Thm]{Corollary}
\newtheorem{Lem}[Thm]{Lemma}

\newtheorem{Defi}[Thm]{Definition}
\newtheorem{Rma}[Thm]{Remark}

\newtheorem{Exmp}[Thm]{Example}
\newtheorem{Th}{Theorem}

\newcommand{\longleft}[1]{\;{\leftarrow%
\count255=0 \loop \mathrel{\mkern-6mu}%
    \relbar\advance\count255 by1\ifnum\count255<#1\repeat}\;}
\newcommand{\longright}[1]{\;{\count255=0 \loop \relbar\mathrel{\mkern-6mu}%
    \advance\count255 by1\ifnum\count255<#1\repeat\rightarrow}\;}

\newcommand{\RIGHT}[3]{\mathrel{\mathop{\kern0pt\longright{#1}}
    \limits^{#2}_{#3}}}

\newcommand{\LEFT}[3]{\mathrel{\mathop{\kern0pt\longleft{#1}}\limits^{#2}_{#3}}
}
\newcommand{\longleftright}[1]{\;{\leftarrow\mathrel{\mkern-6mu}%
    \count255=0\loop\relbar\mathrel{\mkern-6mu}%
    \advance\count255 by1\ifnum\count255<#1\repeat\rightarrow}\;}

\newcommand{\onto}[1]{\;{\count255=0 \loop \relbar\joinrel
    \advance\count255 by1
    \ifnum\count255<#1 \repeat \twoheadrightarrow}\;}

\newcommand{\RLEFT}[3]{\mathrel{%
   \mathop{\vcenter{\baselineskip=0pt\hbox{$\kern0pt\longright{#1}$}%
   \hbox{$\kern0pt\longleft{#1}$}}}\limits^{#2}_{#3}}}

\theoremstyle{Defi}

\newtheorem{Ex}[Thm]{Example}

\def\Top{\curs{Top}}

\def\hocolim{\mathop{\rm hocolim}}

\newcommand{\xxto}[1]{\mathrel{\mathop{%
  \setbox0\hbox{$\ {\scriptstyle#1}\ $}%
  \hbox to \wd0{\rightarrowfill}}^{#1}}%
}

\newcommand{\xlto}[2][]{%
  \mathrel{\mathop{%
    \setbox0\vbox{
      \hbox{$\scriptstyle\;\;{#1}\;\;$}%
      \hbox{$\scriptstyle\;\;{#2}\;\;$}%
    }%
    \hbox to\wd0{\leftarrowfill}\displaystyle}%
  \limits^{#2}\ifx{#1}{}\else{_{#1}}\fi}%
}

\newcommand{\obj}{\mathrm{Obj}}

\newcommand{\inc}{\mathrm{inc}}

\newcommand{\Modk}{\mathbf{Mod}_\Bbbk}

\newcommand{\Ima}{\operatorname{\mathrm{Im}}}
\newcommand{\bXbA}{\mathbf{X}, \mathbf{A}}

\newcommand{\hadgesh}[1]{\emph{\textcolor{blue}{#1}}}

\newcommand{\ini}{*}
\newcommand{\nempty}{\neq\emptyset}
\newcommand{\eempty}{=\emptyset}

\title{Polyhedral Products over Finite Posets}

\author{Daisuke Kishimoto}
\address{Department of Mathematics, Kyoto University, Kyoto, 606-8502, Japan}
\email{kishi@math.kyoto-u.ac.jp}
\author{Ran Levi}
\address{Institute of Mathematics, University of Aberdeen, Aberdeen, UK}
\email{r.levi@abdn.ac.uk}

\subjclass[2010]{Primary 13F55, 55U10, 52B05. Secondary 06A07, 06A12}

\keywords{Finite posets, polyhedral products, Stanley-Reisner rings, Cohomology, Higher limits}

\begin{document}

\begin{abstract}
Polyhedral products were defined by Bahri, Bendersky, Cohen and Gitler, to be spaces obtained as  unions of certain product spaces indexed by the simplices of an abstract simplicial complex. In this paper we give a very general homotopy theoretic construction of polyhedral products over arbitrary pointed posets. We show that under certain restrictions on the poset $\calp$, that include all known cases, the cohomology of the resulting spaces can be computed as an inverse limit over  $\calp$ of the cohomology of the building blocks. This motivates the definition of an analogous algebraic construction - the polyhedral tensor product.  We show that for a large family of posets, the cohomology of the polyhedral product is given by the polyhedral tensor product. We then restrict attention to polyhedral posets, a family of posets that include  face posets of  simplicial complexes, and simplicial posets, as well as many others.   We define the Stanley-Reisner ring of a polyhedral poset and show that, like in the classical cases,  these rings occur as the cohomology of certain polyhedral products over the  poset in question. For any pointed poset $\calp$ we construct a simplicial poset $s(\calp)$, and show that if $\calp$ is a polyhedral poset then polyhedral products over $\calp$ coincide up to homotopy with the corresponding polyhedral products over $s(\calp)$.
\end{abstract}

\baselineskip.525cm

\maketitle



Spaces constructed by basic topological operations from simpler ingredients have played an important role in homotopy theory since the early days of the subject. One example of such a family of space is polyhedral products - spaces obtained by ``gluing together" finite cartesian products of spaces according to a ``recipe" encoded in a simplicial complex. The concept was first defined by Bahri, Bendersky, Cohen, and Gitler \cite{BBCG} as a generalisation of two families of spaces that arose in the context of toric topology, namely the so called Davis-Januszkiewicz spaces and the closely related moment-angle complexes  \cite{DJ,BP}.  Earlier example of similar ideas are present in works of Coxeter \cite{C},  Porter \cite{P} and Anick \cite{A}.

Following these early works, polyhedral products were studies intensely, and were shown to exhibit remarkably regular behavior and relations with many topics in homotopy theory (See for instance \cite{GPTW, GT, HST, IK4, K1, KT, NR, PRV} for a partial list). Polyhedral products are topological spaces that are constructed by gluing together other spaces according to a ``recipe'' given by the combinatorial data in a simplicial complex. As such they are also of potential interest in applied algebraic topology, where one uses real data to generate associated topological objects, the properties of which inform on the generating data. This kind of applications may involved combinatorial objects that are not simplicial complexes, and defining polyhedral products in a more general context is part of the motivation for this paper.

In this paper we define polyhedral products in a more general context. Classically, polyhedral products are defined as colimits  of certain functors over the face poset of a simplicial complex.  Notbohm and Ray \cite{NR} showed that if one uses the homotopy colimit instead of the colimit, the resulting spaces coincide up to homotopy, and one has the additional computational advantage afforded by the Bousfield-Kan spectral sequence  \cite[p.\;336]{BK}. In particular they show that the cohomology of the Davis-Januszkiewicz space, known to be isomorphic to the Stanley-Reisner ring of the complex $K$,  is given by an inverse limit, while the higher derived functors of that inverse limit vanish in this case.

The staring point of our generalisation is the observation that the type of functors used in the standard definition of a polyhedral product  easily generalises to any finite pointed poset (a poset with an initial object $*$).
If $\calp$ is s finite pointed poset, then the \hadgesh{vertex set of $\calp$} is the set $V_\calp$ of all  $v\in\obj(\calp)$, such that  $a\lneq v$, if and only if $a=*$. For such a poset $\calp$ and a collection of pairs of spaces $(\bXbA) = \{(X_v,A_v)\;|\; v\in V_\calp\}$ we define a functor $\calz^{\calp}_{\bXbA}\colon\calp\to\Top$, and the polyhedral product $\calz_\calp(\bXbA)$ of the collection $(\bXbA)$ over $\calp$ is defined to be the homotopy colimit of $\calz^{\calp}_{\bXbA}$ over $\calp$. We show that in the known cases it coincides up to homotopy with existing models.

The definition of the polyhedral product motivates an algebraic analog - a \hadgesh{polyhedral tensor product}. One of the main aims of this paper is to show that under rather general hypotheses the cohomology of a polyhedral product is given by the corresponding polyhedral tensor product. To a finite pointed poset $\calp$, a commutative ring with a unit $\Bbbk$, and a collection of morphisms $\aaa = \{a_v\colon M_v\to N_v\;|\; v\in V_\calp\}$ in $\Modk$,  we associate a functor $\calt_{\calp,\aaa}\colon\calp\op\to \Modk$. The \hadgesh{polyhedral tensor product of $\aaa$ over $\calp$} is defined to be the inverse limit over $\calp$ of the functor $\calt_{\calp,\aaa}$ (Definition \ref{Def-PolyTensorProd}). The main example appears as the 0-th column in the  Bousfield-Kan spectral sequence that computes $H^*(\calz_\calp(\bXbA),\Bbbk)$.

To examine the cases in the cohomology of a polyhedral product is given by  the corresponding polyhedral tensor product we need a vanishing result for higher limits over pointed posets. Our first theorem states that for any finite pointed poset and a certain family of functors on it, this is always the case.

\begin{Th}
\label{Intro-vanishing}
  Let $\calp$ be a finite pointed poset, and let $F\colon \calp\op\to\Ab$ be a functor with a lower factoring section $S$ (see Definition \ref{lower factoring}). Then $F$ is acyclic, i.e.
  \[\higherlim{\calp}{n}F=0\]
  for all $n>0$.
\end{Th}

Theorem \ref{Intro-vanishing} applies to any finite pointed poset, but imposes a nontrivial restriction on the functor $F$. We next show that under certain restrictions on the poset $\calp$, and a mild assumption on the collection $\aaa$ of morphisms in $\Modk$, the functor $\calt_{\calp;\aaa}$ is acyclic.

\begin{Th}
  \label{Intro-T vanishing}
 Let $\calp$ be a lower saturated poset and let $\aaa=\{a_v\colon M_v\to N_v\;|\; v\in V_\calp\}$ be a collection of morphisms in $\Modk$ that has  a section (see Section \ref{Sec Functors with a section}). Then for each $n>0$,
  $$\higherlim{\calp}{n}\calt_{\calp;\aaa}=0.$$
\end{Th}

Theorems \ref{Intro-vanishing} and \ref{Intro-T vanishing}  can be viewed as vanishing results for sheaf cohomology in certain circumstances, where the functors $F$ and $\calt_{\calp, \aaa}$ are considered as sheaves of $\Bbbk$-modules on $\calp$, equipped with the standard Alexandroff topology (compare \cite{B,BR,Y}). Here, as the main application of Theorem \ref{Intro-T vanishing} we obtain a very general result that expresses the cohomology of a polyhedral product as the polyhedral tensor product of the appropriate collection of morphisms.

\begin{Th}
\label{Intro-Thm-chlgy_lower_sat}
  Let $\calp$ be a lower saturated poset, let $(\bXbA)=\{(X_v,A_v)\;|\;v\in V_\calp\}$ be a collection of pairs of spaces, and let $h^*$ be a generalised cohomology theory that satisfies the strong form of the K\"unneth formula. Let $\aaa=\{a_v\colon h^*(X_v)\to h^*(A_v)\;|\; v\in V_\calp\}$ be the collection of maps induced by the inclusions, and assume that $a_v$ has a section for each $v$. Then
  $$h^*(\calz_\calp(\bXbA))\cong \calt_\calp(\aaa).$$
\end{Th}

We next specialise to a very specific type of posets. A finite pointed poset $\calp$ is said to be \hadgesh{polyhedral} if for every object $x\in\obj(\calp)$, the sub-poset $\calp_{\le x}$ is a lower semilattice. In previous work of K. Iriye and the first named author \cite{IK3}, polyhedral products over a lower semilattice were studied in the context of a more general family of spaces, defined as colimits over a lower semilattice. Polyhedral posets are of course more general than lower semilattices, and also include all simplicial posets. Our next theorem is an explicit calculation of the polyhedral tensor product over any polyhedral poset of a family of augmentation morphisms $\eee=\{e_v\colon P_\Bbbk[v]\to\Bbbk\;|\; v\in V_\calp\}$, where $P_\Bbbk[v]$ denotes the ring of polynomials in $v$ over $\Bbbk$.

\begin{Th}
\label{Intro-Thm-PTP-relations}
Let $\calp$ be a finite polyhedral poset. For a subset $S\subseteq\obj(\calp)$, let $[\vee S]$ denote the set of all minimal upper bounds of the set $S$ in $\calp$. Then there is an isomorphism
\[\calt_\calp(\mathbf{e}) \cong P_\Bbbk[\obj(\calp)]/I_\calp,\]
  where the ideal $I_\calp$ is generated by:
  \begin{enumerate}[\rm (a)]
    \item $*-1$; \label{Thm-PTP-relations-a}
    \item $x-y$, if $x<y$ and $V(x)=V(y)$;\label{Thm-PTP-relations-b}
    \item $\prod_{x\in S}x$, for $S\subset \obj(\calp)$ with $[\vee S]=\emptyset$;\label{Thm-PTP-relations-c}
    \item
    $$\prod_{\substack{R\subset S\\|R|\text{ is odd}}}\bigwedge R-\left(\prod_{\substack{R\subset S\\|R|\text{ is even}}}\bigwedge R\right)\cdot\left(\sum_{\substack{z\in[\vee S]\\V(z)=\bigcup_{w\in S}V(w)}}z\right),$$
    for $S\subset \obj(\calp)$ with $[\vee S]\ne\emptyset$. \label{Thm-PTP-relations-d}
  \end{enumerate}
\end{Th}

Motivated by Theorem \ref{Intro-Thm-PTP-relations} we define  the \hadgesh{Stanley-Reisner ring} of a polyhedral poset $\calp$ to be the algebra
\[\Bbbk[\calp] \defeq \calt_\calp(\eee).\]
 This generalises the definition of the Stanley-Reisner ring  of  a simplicial complex \cite[p. 62]{S1} and of
a simplicial poset \cite[Definition 3.3]{S2}, all of  which are particular cases of the Stanley-Reisner ring of a polyhedral poset, as defined above.  Another consequence that follows from Theorems \ref{Intro-Thm-chlgy_lower_sat} and \ref{Intro-Thm-PTP-relations} is that the Stanley-Reisner ring of a polyhedral poset is realisable as the cohomology ring of the polyhedral product $\calz_\calp(\C P^\infty, *)$.

For an arbitrary finite pointed poset $\calp$, we construct a new poset $s(\calp)$, that is a simplicial poset with the same vertex set as $\calp$, together with a poset map $\calp\to s(\calp)$. The poset $s(\calp)$ will be referred to as the \hadgesh{simplicial transform} of $\calp$. If $\calp$ is a polyhedral poset, then under some mild hypotheses we obtain a homotopy equivalence between a polyhedral product over $\calp$ and the polyhedral product over its simplicial transform.

\begin{Th}
	\label{Intro-projection lemma}
	Let $\calp$ be a finite reduced polyhedral poset and $(\bXbA)=\{(X_v,A_v)\;|\; v\in V_\calp\}$ be a collection of NDR-pairs such that $X_v\ne A_v$ for all $v$. Then there is a natural map
	\[\calz_\calp(\bXbA)\to\calz_{s(\calp)}(\bXbA)\]
	that  is a homotopy equivalence.
\end{Th}

We end the paper with a discussion of a special family of  posets. We say that a finite pointed poset is \hadgesh{regular}, if for any two objects $x, y\in \obj(\calp)$, such that $|V(x)| = |V(y)|$, the sub-posets $\calp_{\le x}$ and $\calp_{\le y}$ are isomorphic.  For a  finite regular polyhedral poset $\calp$ we produce a formula that relates the $f$-vector of the simplicial transform $s(\calp)$ and the $f$-vector of $\calp$. One outcome of this is Proposition \ref{f-vector}, which gives an expression for the Poincar\'e series of $\Bbbk[\calp]$ for any finite regular polyhedral poset $\calp$ in terms of its $f$-vector.

The paper is organised as follows. Section \ref{Sec-Prelim} contains some preliminary material. We set our terminology, define polyhedral products and polyhedral tensor products, and recall some known results. In Section \ref{Sec-Acyclic} we recall some basic facts on higher limits. We then define the notion of a lower factoring section for a functor defined on a poset $\calp$, and
prove Theorem \ref{Intro-vanishing} (as Theorem \ref{vanishing}). In Section \ref{Sec-Lower_Sat} we study lower saturated posets, and prove Theorems \ref{Intro-T vanishing} and \ref{Intro-Thm-chlgy_lower_sat} (as Theorems \ref{Intro-T vanishing} and \ref{Thm-chlgy_lower_sat} respectively). We also give a simple example that shows that the lower saturation condition is actually essential to guarantee the vanishing of higher limits. Section \ref{Sec-Polyhedral} is dedicated to polyhedral posets and Stanley-Reisner rings over them. In particular we prove Theorem \ref{Intro-Thm-PTP-relations} (as Theorem \ref{Thm-PTP-relations}). Section \ref{Sec-Colim_vs_Hocolim} is aimed at proving Proposition \ref{hocolim-colim}, which claims that under some mild hypotheses defining  polyhedral products over a finite polyhedral poset as a colimit or a homotopy colimit makes no difference up to homotopy. In Section \ref{Sec-Simplicial_Transform}  we introduce the simplicial transform and prove Theorem \ref{Intro-projection lemma} (as Theorem \ref{projection lemma}). Finally Section \ref{Sec-Regular_Polyhedral_Posets} is dedicated to regular polyhedral posets, their $f$-vectors and some consequences.



\section{Polyhedral products over finite posets}
\label{Sec-Prelim}
 In this preliminary section we set our conventions and terminology, define the main object of study -  polyhedral products over finite pointed posets, and recall some known facts about polyhedral products and associated algebraic structures.


\subsection{Posets}

Throughout this paper we regard a poset $\calp$ as a small category, where the elements of $\calp$ form the object set \hadgesh{$\obj(\calp)$ }of the corresponding category, and for each order relation $x\le y$ in $\calp$, one has a unique morphism \hadgesh{$\iota_{x,y}$}. If $x\leq y$ but $x\neq y$ we may sometime use $x<y$ to express the relation between $x$ and $y$. When referring to objects, we will almost always write \hadgesh{$x\in \calp$} when we actually mean \hadgesh{$x\in\obj(\calp)$} for short.

A poset $\calp$ is said to be \hadgesh{pointed} if it contains an \hadgesh{initial object}, i.e. an object $\ini$ such that $\ini\le x$ for any  $x\in\calp$. We refer to the object $\ini$ as the \hadgesh{base point} of $\calp$. An object $x$ in a pointed poset $\calp$ is said to be \hadgesh{minimal} if $y< x$ implies  $y=\ini$.

\begin{Defi}\label{Def-Vertices}
Let $\calp$ be a finite pointed poset. If $p\in\calp$ is any object, and $v\in\calp$ is a minimal object such that  $v\le p$ in $\calp$, then we say that $v$ is a \hadgesh{vertex} of $p$. The collection of all minimal objects of $\calp$ will be referred to as the \hadgesh{vertex set} of $\calp$ and is denoted $V_\calp$. The collection of vertices of an object $p\in\calp$ is denoted $V_\calp(p)$, or simply $V(p)$ if the ambient poset $\calp$ is fixed.
\end{Defi}

Notice that the base point in a pointed poset is not a vertex.

Let $\calp$ be any poset and let $x, y\in\calp$ be any objects. The \hadgesh{meet} of $x$ and $y$, denoted $x\wedge y$, is the greatest lower bound of $x$ and $y$ in $\calp$. Similarly, the \hadgesh{join} of $x$ and $y$, denoted $x\vee y$, is the least upper bound of $x$ and $y$. If $S$ is a finite set of objects in $\calp$, then one can similarly define the \hadgesh{meet of $S$} and the \hadgesh{join of $S$}, denoted $\wedge S$ and $\vee S$, respectively. Clearly neither join nor meet need exist in general, but if they do then they are unique by definition.

On the other hand, two objects  $x, y\in\calp$  may have one or more lower bounds, and among those there is the subset of \hadgesh{maximal objects}, i.e., those objects $z\in\calp$ such that $z\le x,y$ and if $z'\le x,y$ is another lower bound, then either $z'\le z$, or $z$ and $z'$ are not comparable in $\calp$. Similarly, two objects in $\calp$ may have one or more upper bounds, and among those there is the subset of \hadgesh{minimal objects}. A meet of two or more objects is always a maximal lower bound and a join is always a minimal upper bound. However, a maximal lower bound is a meet if and only if it is unique, and similarly for a minimal upper bound. For a finite set $S$ of objects in $\calp$, let \hadgesh{$[\wedge S]$} and \hadgesh{$[\vee S]$} denote the set of all maximal lower bounds, and minimal upper bounds of $S$, respectively.

For any poset $\calp$ and any object $x\in \calp$, let $\calp_{\le x}$ and $\calp_{\ge x}$ denote the (full) sub-posets of $\calp$ with objects
\[\obj(\calp_{\le x})=\{y\in\calp\,\vert\,y\le x\},\quad\text{and}\quad \obj(\calp_{\ge x})=\{y\in\calp\,\vert\,y\ge x\}.\]


\subsection{Polyhedral product}

We are now ready to define the main object of study in this paper, i.e. polyhedral products over finite pointed posets. As motivation, we recall first the definition of polyhedral products over finite simplicial complexes \cite{BBCG}, and then proceed to introduce our generalisation.

Let $K$ be a simplicial complex with vertex set $\{1,2,\ldots,m\}$ and let $F(K)$ be its face poset. Let $(\bXbA)=\{(X_i,A_i)\}_{i=1}^m$ be a collection of pairs of spaces indexed by vertices of $K$. Define a functor
$\calz^K_{\bXbA}\colon F(K)\to\mathbf{Top}$ by
$$\calz^K_{\bXbA}(\sigma)=\prod_{i=1}^mY_i\quad\text{such that}\quad Y_i=\begin{cases}X_i&i\in\sigma\\A_i&i\not\in\sigma\end{cases}$$
on objects, and if $\sigma\le\tau\in F(K)$, then
$\calz^K_{\bXbA}(\iota_{\sigma,\tau})$ is defined to be the obvious inclusion. The  polyhedral product of $(\bXbA)$ over $K$  is defined by
$$\calz^K(\bXbA)\defeq \colim_{F(K)}\,\calz^K_{\bXbA}.$$
 If  $(X_i,A_i)=(X,A)$ for all $1\le i\le m$, then we denote the polyhedral product of $(\bXbA)$ over $K$ by  by $\calz_K(X,A)$.

The moment-angle complex and the Davis Januszkiewicz space \cite{BP, DJ} for a simplicial complex $K$ are particular examples of the polyhedral products $\calz_K(D^2,S^1)$ and $\calz_K(\C P^\infty,*)$ respectively.

Let $\calp$ be a finite pointed poset with vertex set $V_\calp$. Then  the assignment $x\mapsto V_\calp(x)$  gives rise to a functor from $\calp$ to the power set of $V_\calp$ regarded a poset by the inclusion ordering. On the other hand, the power  set of $V_\calp$ regarded as a poset is isomorphic to the face poset of a simplex on the vertex set $V_\calp$, and one can define  polyhedral products on that face poset, and by composition on any finite pointed poset. However, the colimit may behave badly for general finite pointed posets, and so replacing it by a homotopy colimit makes sense. In particular, if $(\bXbA)$ consists of NDR-pairs and $\calp=F(K)$ is the face poset of a finite simplicial complex $K$, then it follows from \cite[Projection Lemma 1.6]{ZZ} that the natural projection
\[\hocolim_{F(K)}\,\calz_{\bXbA}^{F(K)}\to\colim_{F(K)}\,\calz_{\bXbA}^{K}\] is a homotopy equivalence.
Thus the following definition is a natural generalisation of the standard definition of polyhedral products.

\begin{Defi}
\label{Def-PolyProdPoset}
  Let $\calp$ be a finite pointed poset and $(\bXbA)=\{(X_v,A_v)\;|\; v\in V_\calp\}$ be a collection of pairs of spaces indexed by vertices of $\calp$.
  \begin{enumerate}[{\rm 1)}]
    \item Define a  functor $\calz^\calp_{\bXbA}\colon \calp\to\Top$  by
    $$\calz^\calp_{\bXbA}(x)=\prod_{v\in V_\calp}Y_v,\quad\text{where}\quad Y_v=\begin{cases}X_v&v\in V(x)\\A_v&v\not\in V(x)\end{cases}$$
    on objects, and if  $x\le y\in \calp$, let
   $\calz^\calp_{\bXbA}(\iota_{x,y})$ be the obvious inclusion.

    \item Define the polyhedral product of $(\bXbA)$ over $\calp$  by
    $$\calz_\calp(\bXbA)\defeq \hocolim_{\calp}\,\calz^\calp_{\bXbA}.$$
  \end{enumerate}
\end{Defi}


\subsection{Polyhedral tensor products}
Let $\Bbbk$ be a commutative ring with a unit, and let $\Modk$ the category of $\Bbbk$-modules.
Given a finite pointed poset $\calp$ and a collection of morphisms in $\Modk$ indexed by the vertices of $\calp$, we define a construction that is an algebraic analog of the polyhedral product of spaces. Since the vertices of $\calp$ are not ordered in general, doing so requires a way of constructing a tensor product of modules that is independent of any ordering of the indexing set. In the category of commutative algebras over $\Bbbk$, tensor product is the categorical coproduct, and hence can be defined abstractly as a colimit over a discrete category, without any reference to ordering. However, such description is not available in $\Modk$. This motivates the following definition.

\begin{Defi}\label{Def-unordered_tensor_product}
Let $S$ be a finite set and let $\mathbf{M}=\{M_s\;|\; s\in S\}$ be a collection  of $\Bbbk$-modules indexed by the elements of $S$.

\begin{enumerate}[{\rm 1)}]
\item
Let $\Gamma(S, \mathbf{M})$ denote the set of all sections of the obvious projection $\pi \colon \prod_{s\in S}M_s\to S$.

\item
Let $\Bbbk\Gamma(S,\mathbf{M})$ be the free $\Bbbk$-module generated by $\Gamma(S,\mathbf{M})$. Let $U(S, \mathbf{M})\subseteq \Bbbk\Gamma(S,\mathbf{M})$ be the submodule generated by the following elements:
\begin{enumerate}[{\rm i)}]
\item $\gamma_1 +\gamma_2 - \gamma_3$, where $\gamma_i(s) = \gamma_j(s)$ for all $1\le i, j\le 3$ and all $s\in S$ except some $s_0$ for which $\gamma_3(s_0) = \gamma_1(s_0) + \gamma_2(s_0)$.

\item $a\cdot \gamma_1 - \gamma_2$ for some $a\in\Bbbk$, where $\gamma_1(s) = \gamma_2(s)$ for all $s\in S$ except some $s_0$ for which $\gamma_2(s_0) = a\gamma_1(s_0)$.
\end{enumerate}
In {\rm ii)} above, $a\cdot \gamma_1$ means multiplication of the basis element $\gamma_1\in\Bbbk\Gamma(S,\mathbf{M})$ by $a$, while $a\gamma_1$ means multiplication of the element $\gamma_1(s_0)\in M_{s_0}$ by $a$.
\medskip

Define the \hadgesh{tensor product of $\{M_s\;|\; s\in S\}$} to be the quotient module \[\bigotimes_{s\in S}M_{s}\defeq\Bbbk\Gamma(S,\mathbf{M})/ U(S,\mathbf{M}).\]
\end{enumerate}
\end{Defi}

\begin{Lem}\label{Lem-tensor_generalisation}
Let $S$ be a finite set. Then any ordering $\{s_1,s_2,\ldots, s_n\}$ on the elements of  $S$ determines a canonical isomorphism
\[\bigotimes_{s\in S} M_{s}\cong \bigotimes_{i=1}^n M_{s_i}.
\]
Furthermore, if $\underline{\alpha}=\{\alpha_s\colon M_s\to N_s\;|\; s\in S\}$ is a collection of morphisms in $\Modk$, then there is an induced morphism
\[\bigotimes_{s\in S}\alpha_s\colon \bigotimes_{s\in S}M_s\to \bigotimes_{s\in S}N_s.\]
\end{Lem}
\begin{proof}
Write $M_i$ for $M_{s_i}$ for short. Then the first statement amounts to the obvious identities in the standard tensor prouct:
\[m_1\otimes\cdots\otimes (m'_{i_0}+m''_{i_0})\otimes\cdots \otimes m_n = m_1\otimes\cdots\otimes m'_{i_0}\otimes\cdots\otimes m_n + m_1\otimes\cdots\otimes m''_{i_0}\otimes\cdots m_n,\]
and
\[a(m_1\otimes\cdots\otimes m_{i_0}\otimes\cdots m_n) = m_1\otimes\cdots\otimes am_{i_0}\otimes\cdots m_n.\]

Let $\underline{\alpha}$ be a collection of morphisms in $\Modk$, as above. Then composition with the $\alpha_s$  induce a map $\underline{\alpha}_*\colon\Bbbk\Gamma(S, \mathbf{A})\to \Bbbk\Gamma(S,\mathbf{B})$, where $\mathbf{A} = \{A_s\;|\;s\in S\}$ and $\mathbf{B} = \{B_s\;|\;s\in S\}$.
If $\gamma_i\in \Gamma(S,\mathbf{M})$, $i=1..3$, are elements satisfying the relation i) above, then for each $s\in S$ except $s_0$, $\alpha_s(\gamma_i(s)) = \alpha_s(\gamma_j(s))$ for all $1\le i, j\le 3$, and the relation
\[\alpha_{s_0}(\gamma_3(s_0)) = \alpha_{s_0}(\gamma_1(s_0)) + \alpha_{s_0}(\gamma_2(s_0))\]
holds in $k\Gamma(S,\mathbf{B})$. Similarly one verifies the corresponding statement for the relation ii). Hence $\underline{\alpha}_*(U(S,\mathbf{A})) \subseteq U(S,\mathbf{B})$, and one obtains the induced map $\bigotimes_{s\in S}\alpha_s$, as stated.
\end{proof}

We are now ready to define polyhedral tensor products.

\begin{Defi}
\label{Def-PolyTensorProd}
  Let $\calp$ be a finite pointed poset and let $\mathbf{a}=\{a_v\colon M_v\to N_v\;|\; v\in V_\calp\}$ a collection of morphisms in $\Modk$ indexed by the vertices of $\calp$. For each object $p\in\calp$ and $v\in V(p)$, define
  \[C_{p,v}\defeq\begin{cases} M_v & v\in V(p)\\ N_v & v\notin V(p)\end{cases},\]
  and let
  $\mathbf{C}_p \defeq \{C_{p,v}\;|\; v\in V_\calp\}$.
  \begin{enumerate}[\rm 1)]
    \item Define a functor $\calt_{\calp, \mathbf{a}}\colon \calp\op\to\Modk$  by
    $$\calt_{\calp, \mathbf{a}}(p)\defeq\bigotimes_{v\in V_\calp}C_{p,v}$$
    on objects, and if $p\le q$ let $\calt_{\calp, \mathbf{a}}(\iota_{p,q})\colon \calt_{\calp,\mathbf{a}}(q)\to \calt_{\calp,\mathbf{a}}(p)$ be defined by
    $$\calt_{\calp, \mathbf{a}}(\iota_{p,q})\defeq\bigotimes_{v\in V_\calp}f_v,\quad\text{where}\quad
     f_v=\begin{cases}1_{M_v}&v\in V(p)\\\alpha_v&v\in V(q)\setminus V(p)\\1_{N_v}&v\not\in V(p).\end{cases}$$
  \end{enumerate}
Define the \hadgesh{polyhedral tensor product of $\mathbf{a}$ over $\calp$} by
\[\calt_\calp(\mathbf{a})\defeq \lim_\calp\calt_{\calp,\mathbf{a}}.\]
\end{Defi}

An important example of  polyhedral tensor products arises naturally as the cohomology of polyhedral products. Let $\calp$ be a finite pointed poset and $(\bXbA)=\{(X_v,A_v)\;|\; v\in V_\calp\}$ be a collection of pairs of spaces, indexed by $V_\calp$. Let $\Bbbk$ be a field, and let $h_*$ be a cohomology theory that takes values in $\Bbbk$-vector spaces, and which satisfies the strong form of the K\"unneth formula:
\[h^*(X\times Y) \cong h^*(X)\otimes h^*(Y).\]
For each $v\in V_\calp$, set $M_v = h^*(X_v)$, $N_v =  h^*(A_v)$ and let $a_v\colon M_v\to N_v$ be the map induced on cohomology by the inclusion $A_v\subseteq X_v$.
Let  $\calt_{\calp, \mathbf{a}}\defeq h^*\circ\calz_{\bXbA}^\calp$.
Thus
\begin{equation}
  \label{T Z}
  \calt_{\calp}(\mathbf{a})\cong  \lim_{\calp} (h^*\circ\calz_{\bXbA}^\calp).
\end{equation}
We will show that under certain circumstances $h^*(\calz_\calp(\bXbA))$ is given by the polyhedral tensor product $\calt_\calp(\mathbf{a})$ (See Theorem \ref{Thm-chlgy_lower_sat}).


\subsection{Stanley-Reisner ring}

Let $K$ be a simplicial complex with vertices $v_1,\ldots,v_m$. The \hadgesh{Stanley-Reisner ring} for $K$ is defined by
$$\Bbbk[K]=\Bbbk[v_1,\ldots,v_m]/(v_{i_1}\cdots v_{i_k}\,\vert\,\{v_{i_1},\ldots,v_{i_k}\}\not\in K).$$
Davis and Januszkiewicz \cite[Theorem 4.8]{DJ} proved that there is an isomorphism
\begin{equation}
  \label{k[K] DJ}
 H^*(\calz_K(\C P^\infty,*))\cong\Bbbk[K].
\end{equation}
Notbohm and Ray \cite[Section 3]{NR} considered  $\calz_K(\C P^\infty,*)$ as a homotopy colimit, and proved that
\begin{equation}\label{Eq-NR}
H^n(F(K);H^*\circ \calz^K_{\C P^\infty,*})\cong\begin{cases}\Bbbk[K]&n=0\\0&n\ge 1,\end{cases}.
\end{equation}
Thus, using the Bousfield-Kan spectral sequence \cite[p.\;336]{BK}, they recovered the isomorphism \eqref{k[K] DJ}.
 Hence their result can be stated as saying that the cohomology of the polyhedral product is given by the polyhedral tensor product.

Recall that $\calp$ is a \hadgesh{simplicial poset} if for each $x\in\calp$, the sub-poset $\calp_{\le x}$ is isomorphic to a boolean algebra, i.e. the face poset of a simplex. Obviously, the face poset of a simplicial complex is a simplicial poset, but the converse is false. For example, if $\calp$ is a poset with the  Hasse diagram:
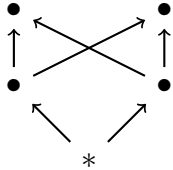
\begin{center}
  \begin{tikzpicture}[thick]
    \node            	(bp) at (0,0)  { $*$ };
    \node            	(a) at (-1,1)  { $\bullet$ };
    \node            	(b) at (1,1)  { $\bullet$ };
    \node            	(c) at (-1,2)  { $\bullet$ };
    \node            	(d) at (1,2)  { $\bullet$ };
      \draw [->] (bp) edge (a);
    \draw [->] (bp) edge (b);
    \draw [->] (a) edge (c);
    \draw [->] (a) edge (d);
     \draw [->] (b) edge (c);
    \draw [->] (b) edge (d);
  \end{tikzpicture}
\end{center}
 then it is a simplicial poset but  not the face poset of any simplicial complex.

For any set $X$, let $P_\Bbbk[X]$ denote the polynomial algebra over $\Bbbk$ generated by a set $X$. Stanley \cite{S2} generalised the definition of the Stanley-Reisner ring of a simplicial complex to simplicial posets.

\begin{Defi}[{\cite[Definition 3.3]{S2}}]
\label{Def-SR-ring-Simplicial-Poset}
Let $\calp$ be a finite simplicial poset. The Stanley-Reisner ring for $\calp$ over $\Bbbk$ is defined by
$$\Bbbk[\calp]=P_\Bbbk[\obj(\calp)]/I_\calp,$$
where the ideal $I_\calp$ is generated by the following elements:
\begin{enumerate}[\rm (i)]
  \item $\ini-1$; \label{Def-SR-ring-Simplicial-Poset-i}
  \item $xy$, if $[x\vee y]\eempty$;\label{Def-SR-ring-Simplicial-Poset-ii}
  \item $xy-(x\wedge y)\sum_{z\in[x\vee y]} z$, if  $[x\vee y]\nempty$. \label{Def-SR-ring-Simplicial-Poset-iii}
\end{enumerate}
\end{Defi}

Part iii) of the definition of $I_\calp$ uses the observation that if $\calp$ is  simplicial poset and  $x,y\in P$ have a common upper bound, then they have a meet  $x\wedge y$. The proof is elementary

If $\calp=F(K)$ is the face poset of a simplicial complex $K$, then it follows easily that the map $V_\calp\to\Bbbk[\calp]$ extends to an isomorphism $\Bbbk[K]\to\Bbbk[\calp]$, and so Stanley-Reisner rings of simplicial posets are a generalisation of the corresponding construction for simplicial complexes. We will see below (Corollary \ref{Cor-generalization-of-simp}) that if $\calp$ is a simplicial poset, then $H^*(\calz_\calp(\C P^\infty,*))$ is isomorphic to the Stanley-Reisner ring for $\calp$, which in turn is the polyhedral tensor product $\calt_\calp(\mathbf{e})$, where $\mathbf{e} = \{e_v\;|\;v\in V_\calp\}$, and $e_v\colon H^*(\C P^\infty, \Bbbk)\to \Bbbk$ is the augmentation for all $v\in V_\calp$. This was proven by L\"u and Panov in \cite[Lemma 2.5 and Remark 3.2]{LP}.


\section{Acyclic functors on finite posets}
\label{Sec-Acyclic}
In this section we study functors defined on arbitrary pointed posets from the point of view of their higher limits, or equivalently functor cohomology. The aim of the section is to characterise a family of functors that are acyclic, i.e., for which all higher limits vanish.


\subsection{Functor cohomology and the derived functors of the inverse limit}

We start by recalling  the definition of higher limits and some basic properties. Let $\calp$ be a small category, let $\Ab$ be the category of abelian groups, and let $F\colon \calp\op\to\Ab$ be a  functor. Recall that the \hadgesh{cohomology of $\calp$ with coefficients in $F$} can be computed as the cohomology of the cochain complex defined as follows:
\[C^n(\calp, F)=\prod_{x_0\to\cdots\to x_n}F(x_0),\]
with a differential  $\delta\colon C^{n}(\calp, F)\to C^{n+1}(\calp, F)$ given by
\[
\begin{split}
(\delta c)(x_0\to\cdots\to x_{n+1})=
&F(x_0\to x_1)(c(x_1\to\cdots\to x_{n+1})) +\\
&\sum_{k=1}^i(-1)^kc(x_0\to\cdots\to\widehat{x_k}\to\cdots\to x_{n+1})
\end{split}\]
for $c\in C^n(\calp, F)$. Then
\[H^n(\calp, F) = H^n(C^*(\calp, F), \delta).\]
By definition
\[H^0(\calp, F) \cong \lim_\calp F,\quad\text{and for $n>0$}\quad H^n(\calp, F) \cong \higherlim{\calp}{n} F,\]
the derived functors of the inverse limit, frequently referred to as \hadgesh{the higher limits of $F$}. In particular, if $F$ is a constant functor on $\calp$ with value $A\in\Ab$, then $H^*(\calp, F) \cong H^*(B\calp, A)$, where $B\calp$ denotes the geometric realisation of the nerve of $\calp$.

We proceed with some observations on higher limits that are relevant to our context.
Recall that a sub-poset $\calq\subseteq\calp$ is said to be an \hadgesh{upper set} in $\calp$, if $y\in\calq$ implies $x\in\calq$ for all $x\geq y$. For $x\in \calp$,  the sub-poset  $\calp_{\ge x}$ is an upper set, and every upper set in $\calp$ is the union of $\calp_{\ge x}$. More generally one has the following.

\begin{Defi}\label{Def-Closed_Upwards}
Let $\calp$ be a category and $\calq\subseteq\calp$ a subcategory. We say that $\calq$ is \hadgesh{closed upwards} if for each $y\in \calq$ and $x\in\calp$, the $\Mor_\calp(y,x)\neq \emptyset$ implies $x\in\calq$. If $S\subseteq\obj(\calp)$ is a collection of objects, then the \hadgesh{upwards closure of $S$ in $\calp$} is the full subcategory of $\calp$ whose objects are all those $x\in\calp$ such that $\calp(y,x)\neq\emptyset$ for some $y\in S$.
\end{Defi}

Subcategories of a small category that are closed upwards have the following pleasant property.

\begin{Lem}
 \label{restriction higher limit}
Let $\calp$ be a small category and let  $\calq\subseteq\calp$ be a subcategory that is closed upwards. Let $\cala$ be an abelian category and let $F\colon \calp^{\op}\to \cala$ be a functor.
Assume that $F(p) = 0$ for all $p\in \calp\setminus \calq$. Then for all $n\geq 0$,
\[\higherlim{\calp}{n} F\cong \higherlim{\calq}{n} F|_\calq.\]
\end{Lem}
\begin{proof}
By Definition
$$C^n(\calp;F)\defeq\prod_{x_0\to\cdots\to x_n}F(x_0).$$
Since $\calq$ is closed upwards, $C^n(\calp;F)$ can be split as a product
\[C^n(\calp;F)=\prod_{\stackrel{x_0\to\cdots\to x_n}{x_0\in\calp\setminus\calq}}F(x_0)\times \prod_{\stackrel{x_0\to\cdots\to x_n}{x_0\in\calq}}F(x_0) = \prod_{\stackrel{x_0\to\cdots\to x_n}{x_0\in\calp\setminus\calq}}F(x_0) \times C^n(\calq, F|_\calq).\]
But the first factor vanishes by assumption, and hence the obvious inclusion
\[C^*(\calq, F|_\calq) \xto{\inc} C^*(\calp, F)\]
is an isomorphism of cochain complexes. The claim follows.
\end{proof}

\begin{Prop}
  \label{A_x}
Fix an abelian group $A$, a poset $\calp$ and an object $x\in \calp$. Define a functor $A_x\colon \calp^\mathrm{op}\to\Ab$ by
$$A_x(y)=\begin{cases}A&y\in \calp_{\ge x}\\0&y\not\in \calp_{\ge x}\end{cases}\quad\text{and}\quad A_x(y\le z)=\begin{cases}1_A&y\in \calp_{\ge x}\\0&y\not\in \calp_{\ge x}.\end{cases}$$
Then  $\higherlim{\calp}{i}A_x=0$ for all $i>0$.
\end{Prop}

\begin{proof}
Since $\calp_{\geq x}$ is an upper set in $\calp$, away from which $A_x$ vanishes,
\[\higherlim{\calp}{i}A_x\cong\higherlim{\calp_{\geq x}}{i}A_x\vert_{\calp_{\ge x}}\]
for all $i\geq 0$ by Lemma \ref{restriction higher limit}. Since $A_x\vert_{P_{\ge x}}$ is the constant functor with value $A$,
\[\higherlim{\calp_{\ge x}}{i}A_x\vert_{\calp_{\ge x}}\cong H^i(B\calp_{\ge x}, A).\] But since $\calp_{\ge x}$ has an initial object $x$,  $B\calp_{\ge x}$ is contractible, and so $H^i(B\calp_{\ge x};A)=0$ for $i>0$.
\end{proof}

Let $\calp$ be any small category, and let $F, G\colon \calp\op\to\Ab$ be functors. We say that $G$ is a subfunctor of $F$, and denote this relation by $G\le F$, if for each object $x\in\calp\textcolor{red}{}$, $G(x)\le F(x)$ and the inclusion forms a natural transformation $G\to F$. If $G\le F$, then the quotient functor $F/G\colon \calp\op\to\Ab$ is defined by $(F/G)(x)=F(x)/G(x)$ with the obvious effect on morphisms. There is a canonical projection transformation $F\to F/G$ and an exact sequence of functors $0\to G\to F\to(F/G)\to 0$. The following lemma is elementary.

\begin{Lem}
  \label{subfunctor_exact_sequence}
 Let $\calp$ be a small category, let $F\colon\calp\op\to\Ab$ be a functor, and let $G\le F$ be a subfunctor. Then there is a long exact sequence
 \[\cdots\to\higherlim{\calp}{i}G\to\higherlim{\calp}{i}F\to\higherlim{\calp}{i}F/G\to\higherlim{\calp}{i+1}G\to\cdots.\]
\end{Lem}


\subsection{Lower factoring sections}
\label{Sec Functors with a section}

Let $\calp$ be a poset and let $F\colon\calp\op\to\Ab$ be a functor. We say that $F$ admits a \hadgesh{section} if there is a functor $S\colon \calp\to\Ab$ satisfying the following conditions:
\begin{enumerate}[{\rm(a)}]
  \item $S(x)=F(x)$ for all $x\in \calp$;
  \item For all $x\le y\in P$, $S(\iota_{x,y})$ is a right inverse for $F(\iota_{x,y})$, i.e., $F(\iota_{x,y})\circ S(\iota_{x,y})=1_{F(x)}$.
\end{enumerate}

We are interested in the higher limits of certain functors, and in particular in conditions that will guarantee that the higher limits vanish. In general the higher limits of functors with a section are not trivial. Indeed, the simplest example of a functor with a section is any constant functor. But the higher limits of such a functor on a poset $\calp$ are the cohomology groups of $B\calp$ with coefficients in the value of the functor, and these do not vanish in general. Thus we restrict attention to functors on a  that satisfy a stronger requirement.

\begin{Defi}
\label{lower factoring}
Let $\calp$ be a pointed poset, and let $F\colon\calp\op\to\Ab$ be a functor. We say that $F$ has a \hadgesh{lower factoring section} $S$, if $S$ is a section for $F$ and, in addition, if for any pair of objects $x,y\in P$ that have an upper bound  $z\ge x, y$,  there is a lower bound $w\le x, y$ such that the diagrams
\begin{equation}
  \label{lower factoring diagram}
  \xymatrix{
  & F(z)\ar[dr]^{F(\iota_{y,z})}\ar@{<-}[dl]_{S(\iota_{x,z})}&&&	& F(z)\ar@{<-}[dr]^{S(\iota_{y,z})}\ar[dl]_{F(\iota_{x,z})}	\\
  F(x)\ar[dr]_{F(\iota_{w,x})} && F(y)\ar@{<-}[dl]^{S(\iota_{w,y})}&&	F(x)\ar@{<-}[dr]_{S(\iota_{w,x})} && F(y)\ar[dl]^{F(\iota_{w,y})}	\\
  &F(w)&&&											&F(w)}
\end{equation}
commute.
\end{Defi}

In Definition \ref{lower factoring} the factorisation condition must hold for any pair $x, y$. Hence it suffices to require that one of the squares in Diagram (\ref{lower factoring diagram}) commutes, as commutativity of the other would follow by symmetry.

\begin{Lem}
  \label{lower factoring minimal and quotient}
  Let $\calp$ be a pointed poset, and let $F\colon\calp\op\to\Ab$ be a functor with a lower factoring section $S$. Then the following statements hold:
  \begin{enumerate}[{\rm (a)}]
  \item Let $u\in\calp$ be a minimal object, let $\calp_0\subseteq \calp$ be the full sub-poset on all objects but $u$, and assume $F(u) = F(*) =0$. Then $S\vert_{\calp_0}$ is a lower factoring section of $F\vert_{\calp_0}$.
  \label{lower factoring minimal and quotient a}
  \item If $F_0\le F$ is a subfunctor such that $S$ restricts to a lower factoring section $S_0\le S$ for $F_0$, then $S/S_0$ is a lower factoring section for $F/F_0$.
  \label{lower factoring minimal and quotient b}
  \end{enumerate}
\end{Lem}

\begin{proof}
  Clearly, $S\vert_{\calp_0}$ is a section of $F\vert_{\calp_0}$, so to prove part (\ref{lower factoring minimal and quotient a}) it remains to check that $S\vert_{\calp_0}$ is lower factoring. By assumption, given $x,y\in \calp_0$ and an upper bound $z\ge x, y$, there is a lower bound $w\in \calp$ for $x$ and $y$, such that the  diagrams \eqref{lower factoring} above commute. If $w\ne u$, then these diagrams are contained in $\calp_0$. On the other hand, if $w=u$, then $F(u)=0$, and so
  \[F(\iota_{y,z})\circ S(\iota_{x,z}) = 0 = F(\iota_{x,z})\circ S(\iota_{y,z}).\]
  Hence we may set $w=*$, which is an object in $\calp_0$ by definition, and clearly satisfies the requirement.

 Since $S_0$ is a subfunctor of $S$, we have a functor $S/S_0$ which is obviously a section of $F/F_0$. Commutativity of Diagrams \eqref{lower factoring}  for $F/F_0$ and $S/S_0$ follows at once, and so $S/S_0$ is a lower factoring section of $F/F_0$. This proves Part \eqref{lower factoring minimal and quotient b}.
\end{proof}


\subsection{Acyclicity of functors with lower factoring sections}

The following lemma is the key to prove the main theorem of this section.

\begin{Lem}
  \label{locally constant}
  Let $\calp$ be a pointed poset and let $F\colon\calp\op\to\Ab$ be a functor that admits a lower factoring section $S$ and such that $F(*)=0$. Let $u\in\calp$ be a minimal object. Then there is a subfunctor $F_u\le F$ that is locally constant on $\calp_{\ge u}$, with  $F_u(x)\cong F(u)$ for all $x\in\calp_{\ge u}$, and $F_u(y)=0$ for all $y\notin \calp_{\geq u}$. Furthermore, the section $S$ restricts to a lower factoring section $S_u$ of $F_u$.
\end{Lem}

\begin{proof}
  Define $F_u$ by
  $$F_u(x)=\begin{cases}S(\iota_{u,x})(F(u))&x\in \calp_{\ge u}\\0&x\not\in \calp_{\ge u}\end{cases}\quad\text{and}\quad F_u(\iota_{x,y})=\begin{cases}F(\iota_{x,y})\vert_{F_u(y)}&x\in \calp_{\ge u}\\0&x\not\in \calp_{\ge u}.\end{cases}$$
  This is well-defined since
  \begin{align*}
  F_u(\iota_{x,y})(F_u(y)) &= F(\iota_{x,y})\circ S(\iota_{u,y})(F(u))\\
  &= F(\iota_{x,y})\circ S(\iota_{x,y})\circ S(\iota_{u,x})(F(u))\\
  &= S(\iota_{u,x})(F(u))\\
	&= F_u(x).
\end{align*}
  For each object $x\in\calp$, let $\inc_x\colon F_u(x)\to F(x)$  denote the obvious inclusion. We claim that $\inc\colon F_u\to F$ is a  natural transformation. Thus, we must show that for any $x\le y\in \calp$,  the  diagram
  \begin{equation}
    \label{F_u}
    \xymatrix{F_u(y)\ar[rr]^{F_u(\iota_{x,y})}\ar[d]_{\inc_y}&&F_u(x)\ar[d]^{\inc_x}\\
    F(y)\ar[rr]^{F(\iota_{x,y})}&&F(x)}
  \end{equation}
commutes. There are  three cases to consider
\begin{enumerate}[(a)]
\item $x\in P_{\ge u}$,
\item $x\not\in P_{\ge u},y\in P_{\ge u}$, and
\item $x,y\not\in P_{\ge u}$.
\end{enumerate}
In the case (a), $y\in\calp_{\ge u}$, and  one has
\[F(\iota_{x,y})\circ S(\iota_{u,y}) =  F(\iota_{x,y})\circ S(\iota_{x,y})\circ S(\iota_{u,x}) = S(\iota_{u,x}).\]
Commutativity of Diagram \eqref{F_u}  follows from the definition of $F_u$ in this case. In case (c),  $F_u(x) = F_u(y) = 0$, and the diagram commutes trivially. It remains to prove commutativity in case (b). In this case,  $F_u(x)=0$ so it suffices to show that $F_u(\iota_{x,y})=0$. Since $S$ is a lower factoring section, there is some $v\in \calp$ such that $v\le x, u$, and $F(\iota_{x,y})\circ S(\iota_{u,y})=S(\iota_{v,x})\circ F(\iota_{v,u})$. Thus
  \[F_u(\iota_{x,y})(F_u(y))=F(\iota_{x,y})\circ S(\iota_{u,y})(F(u))=S(\iota_{v,x})\circ F(\iota_{v,u})(F(u)).\]
  Since $x\not\in P_{\ge u}$,  it follows that $v\lneq u$, and hence that $v=*$ by the minimality of $u$. Since $F(*)=0$ by assumption, it follows that  $F(\iota_{v,u})=0$, and hence that $F_u(\iota_{x,y})(F_u(y))=0$, as desired.

 By definition $F_u(x) \cong F(u)$ for all $x\in \calp_{\ge u}$. Furthermore,  for any $x\le y$ in $\calp_{\ge u}$ one has
 \[F_u(\iota_{x,y}) \defeq F(\iota_{x,y})|_{F_u(y)} = F(\iota_{x,y})\circ S(\iota_{u,y})|_{F(u)} = F(\iota_{x,y})\circ S(\iota_{x,y})\circ S(\iota_{u,x})|_{F(u)} =S(\iota_{u,x})|_{F(u)}.\]
 This shows that $F_u(\iota_{x,y})$ is a monomorphism, and hence an isomorphism for all $\iota_{x,y}$ in $\calp_{\ge u}$. This shows that $F_u$ is locally constant on $\calp_{\ge u}$.

Finally, we must show that $S$ restricts to a lower factoring section $S_u$ for $F_u$. Let $S_u$ denote the restriction of $S$ to the values of the subfunctor $F_u$. This clearly define a section for $F_u$, and it remains to show that it is lower factoring. For $x,y\in\calp_{\ge u}$, one can take $w=u$ and verify that the diagrams \eqref{lower factoring} commute. In all other cases, either $F_u(x)$ or $F_u(y)$ vanish and one may take $w=*$ to satisfy commutativity in \eqref{lower factoring}. This shows that $S_u$ is a lower factoring section, and hence completes the proof.
\end{proof}

We are now ready to state and prove the main theorem of this section.

\begin{Thm}
\label{vanishing}
  Let $\calp$ be a finite pointed poset, and let $F\colon \calp\op\to\Ab$ be a functor with a lower factoring section $S$. Then $F$ is acyclic, i.e.
  \[\higherlim{\calp}{n}F=0\]
  for all $n>0$.
\end{Thm}

\begin{proof}
Define a functor $F_0\colon \calp\op\to\Ab$ by $F_0(x)=S(\iota_{*,x})(F(*))$ and $F_0(\iota_{x,y})=F(\iota_{x,y})\vert_{F_0(y)}$.
This is well-defined because
\[F(\iota_{x,y})(F_0(y))=F(\iota_{x,y})(S(\iota_{*,y})(F(*)))=S(\iota_{*,x})(F(*))=F_0(x).\]
Then $F_0\le F$, and by Lemma \ref{subfunctor_exact_sequence}, there is a long exact sequence
  $$\cdots\to\higherlim{\calp}{i}F_0\to\higherlim{\calp}{i}F\to\higherlim{\calp}{i}F/F_0\to\higherlim{\calp}{i+1}F_0\to\cdots.$$
  Since $F_0$ is isomorphic to the constant functor on $\calp=\calp_{\ge *}$, we have $\higherlim{\calp}{i}F_0=0$ for $i>0$ by Proposition \ref{A_x}, and so $\higherlim{\calp}{i}F\cong\higherlim{\calp}{i}F/F_0$ for $i>0$.

  Set  $G\defeq F/F_0$, for short. It is easy to verify that that $S$ restricts to a lower factoring section $S_0$ of $F_0$. Thus by Lemma \ref{lower factoring minimal and quotient}(b), $G$ has a lower factoring section, and by consruction $G(*)=0$. Thus by Lemma \ref{locally constant}, for any minimal object $u\in\calp$, there is a subfunctor $G_u\le G$ that is locally constant on $\calp_{\ge u}$ with value $G(u)$ there, and such that $G_u$ vanishes away from $\calp_{\ge u}$. By Lemma \ref{subfunctor_exact_sequence} there is a long exact sequence
  $$\cdots\to\higherlim{\calp}{i}G_u\to\higherlim{\calp}{i}G\to\higherlim{\calp}{i}G/G_u\to\higherlim{\calp}{i+1}G_u\to\cdots$$
  and by Proposition \ref{A_x}, we have $\higherlim{\calp}{i}G_u=0$ for $i>0$. Thus  $\higherlim{\calp}{i}G\cong\higherlim{\calp}{i}G/G_u$ for $i>0$.

  By Lemma \ref{locally constant}, a lower factoring section of $G$ restricts to a lower factoring section of $G_u$, so by Lemma \ref{lower factoring minimal and quotient}(b), $G/G_u$ has a lower factoring section. Let $\calp_1$ be the full sub-poset on all objects but $u$, and let $F_1$ denote $(G/G_u)\vert_{\calp_1}$. Then since $(G/G_u)(*)=(G/G_u)(u)=0$, it follows from Lemmas \ref{restriction higher limit} and \ref{lower factoring minimal and quotient}(a) that
  \[\higherlim{\calp}{i}G/G_u\cong \higherlim{\calp_1}{i}F_1.\]
  for $i\ge 0$ and $F_1$ has a lower factoring section. This shows that for all $i>0$,
  \[\higherlim{\calp}{i} F\cong \higherlim{\calp}{i} G \cong \higherlim{\calp}{i}G/G_u \cong \higherlim{\calp_1}{i}F_1.\]

Let $n+1$ be the cardinality of the object set of $\calp$. Thus $\calp$ contains $n$ objects different from the base point. Applying the procedure above $k$ times one obtains a pointed poset $\calp_k$ with $n-k+1$ objects, and a functor $F_k\colon\calp_k\to \Ab$ that admits a lower factoring section, and such that for all $i>0$,
 \[\higherlim{\calp}{i} F\cong \higherlim{\calp_k}{i}F_k.\]
This holds in particular for $k=n$, and since $\calp_n$ is the trivial pointed poset, and $\higherlim{\calp_n}{i}F_n =0$ for all $i>0$. This completes the proof.
\end{proof}



\section{Lower saturated posets}
\label{Sec-Lower_Sat}

Theorem \ref{vanishing} gives a condition on a functor defined on an arbitrary pointed poset $\calp$ that ensure its higher derived limits vanish. In this section we consider  conditions on the poset $\calp$ and a collection of morphisms $\aaa = \{a_v\colon M_v\to N_v\;|\; v\in V_\calp\}$ in $\Modk$, indexed by the vertices of $\calp$, such that the polyhedral tensor product functor $\calt_{\calp, \aaa}$ has a lower factoring section. In this setup Theorem \ref{vanishing} applies, and the higher limits of the functor $\calt_{\calp, \mathbf{a}}$  vanish. This is then used to show that under certain very general conditions the cohomology of a polyhedral product is given by a polyhedral tensor product of the cohomology of the factors.

Let $\calp$ be a poset and $\aaa=\{a_v\colon M_v\to N_v\;|\; v\in V_\calp\}$ be a collection of morphisms in $\Modk$.  A \hadgesh{section} of $\aaa$ is a collection of morphisms $\sss=\{s_v\colon N_v\to M_v\;|\; v\in V_\calp\}$ such that  for each $v\in V_\calp$, $a_v\circ s_v = 1_{N_v}$. Given a section $\sss=\{s_v\colon N_v\to M_v\;|\; v\in V_\calp\}$ for $\aaa$, define a functor $\cals_{\calp, \sss}\colon \calp\to\Modk$ by $\cals_{\calp, \sss}(x)=\calt_{\calp,\aaa}(x)$ on objects, and if $x\le y$, then
$$\cals_{\calp,\sss}(\iota_{x,y})=\bigotimes_{v\in V_\calp}g_v,\quad\text{where}\quad g_v=\begin{cases}1_{M_v}&v\in V(x) \\s_v&v\in V(y)\setminus V(x)\\1_{N_v}&v\not\in V(y)\end{cases}$$
for $x\le y\in P$. Then $\cals_{\calp, \sss}$ is clearly a section of $\calt_{\calp,\aaa}$. Next, we consider a condition on $\calp$ which ensures that the section $\cals_{\calp, \sss}$ is lower factoring.

\begin{Defi}
\label{Def-lower_factoring}
  We say that a poset $P$ is \hadgesh{lower saturated} if for each pair of objects $x,y\in \calp$ such that $[x\vee y]\neq\emptyset$,  there is an object $w\in [x\wedge y]$, such that  $V(w)=V(x)\cap V(y)$.
\end{Defi}

For instance, simplicial posets are lower saturated: if $\calp$ is a simplicial poset and $x,y\le z\in \calp$, then the map $\calp_{\le z}\to 2^{V(z)}$ that sends and object $a$ to $V(a)$ is a poset isomorphism. Hence, there is $w\in \calp_{\le z}$ such that $w\le x,y$ and $V(w)=V(x)\cap V(y)$.

\begin{Lem}
  \label{lower saturated}
  Let $\calp$ be a poset and $\aaa=\{a_v\colon M_v\to N_v\;|\; v\in V_\calp\}$ be a collection of morphisms in $\Modk$. If $\calp$ is lower saturated and $\aaa$ has a section $\sss$, then $\cals_{\calp,\sss} = \{s_v\colon N_v\to M_v\;|\; v\in V_\calp\}$ is a lower factoring section of of $\calt_{\calp, \aaa}$.
\end{Lem}

\begin{proof}
  Suppose $x,y\le z\in \calp$. Since $\calp$ is lower saturated, there is $w\in \calp$ such that $w\le x,y$ and $V(w)=V(x)\cap V(y)$. Then
  $$\calt_{\calp,\aaa}(\iota_{y,z})\circ \cals_{\calp,\sss}(\iota_{x,z})=\bigotimes_{v\in V_\calp}f_v=\cals_{\calp,\sss}(\iota_{w,y})\circ \calt_{\calp,\aaa}(\iota_{w,x}),$$
  where
  $$f_v=\begin{cases}1_{M_v}&v\in V(x)\cap V(y)\\a_v&v\in V(x)\setminus V(y)\\s_v&v\in V(y)\setminus V(x)\\1_{N_v}&v\not\in V(x)\cup V(y).\end{cases}$$
  Thus $\cals_{\calp,\sss}$ is  a lower factoring section, as claimed.
\end{proof}

As an immediate corollary of Theorem \ref{vanishing} and Lemma \ref{lower saturated}, one gets:

\begin{Thm}
  \label{T vanishing}
 Let $\calp$ be a lower saturated poset and let $\aaa=\{a_v\colon M_v\to N_v\;|\; v\in V_\calp\}$ be a collection of morphisms in $\Modk$ that has  a section. Then for each $n>0$,
  $$\higherlim{\calp}{n}\calt_{\calp,\aaa}=0.$$
\end{Thm}

\begin{Cor}
  \label{T vanishing surjection}
 Let $\calp$ be a lower saturated poset, let $\aaa=\{a_v\colon M_v\to N_v\;|\; v\in V_\calp\}$ be a collection of surjective morphisms in $\Modk$, where $\Bbbk$ is a field. Then for all $n>0$,
  $$\higherlim{\calp}{n}\calt_{\calp,\aaa}=0.$$
  In particular, if $\{M_v\;|\; v\in V_\calp\}$ is a collection of augmented modules in $\Modk$, then $\higherlim{\calp}{n}\calt_{\calp, \eee}=0$, where $\eee$ is the corresponding collection of augmentations.
\end{Cor}

We proceed with a very simple example that shows that the lower saturation condition we impose on $\calp$ is necessary for the vanishing of higher limits.

\begin{Ex}\label{non-vanishing}
Let $\calp$ be the poset given by the following Hasse diagram:
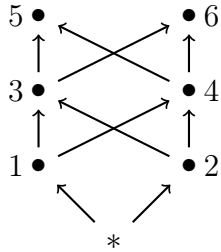
\begin{center}
  \begin{tikzpicture}[thick]
 \node	(bp) 	at (1,0)	{ $*$ };
 \node	(a)	at (0,1)	{ $\bullet$ };
 \node	(b)	at (0,2)	{ $\bullet$ };
 \node	(c)	at (0,3)	{ $\bullet$ };
 \node	(d)	at (2,1)	{ $\bullet$ };
 \node	(e)	at (2,2)	{ $\bullet$ };
 \node	(f)	at (2,3)	{ $\bullet$ };
  \node	(aa) at 	(-0.3,1)	{\rm 1};
  \node	(bb) at 	(2.3,1) 	{\rm 2};
  \node	(cc)	at	(-0.3,2)	{\rm 3};
  \node	(dd)	at	(2.3,2)	{\rm 4};
  \node	(ee)	at	(-0.3,3)	{\rm 5};
  \node	(ff)	at	(2.3,3)	{\rm 6};
 \draw [->] (bp) edge (a);
 \draw [->] (bp) edge (d);
 \draw [->] (a) edge (b);
 \draw [->] (a) edge (e);
 \draw [->] (b) edge (c);
 \draw [->] (b) edge (f);
\draw [->] (d) edge (e);
 \draw [->] (d) edge (b);
 \draw [->] (e) edge (c);
 \draw [->] (e) edge (f);
  \end{tikzpicture}
\end{center}
Then $\calp$ is not lower saturated.
Furthermore, let $F\colon\calp\op\to\Ab$ be the functor defined by assigning the trivial group to the objects $*, 1, 2$ and $F(j) = \Z$ for $j=3\ldots 6$, with $F(\iota_{k.l}) = 1_\Z$ for $3\le k\le l$. Then $\higherlim{\calp}{1}\, F\neq 0$.
\end{Ex}
\begin{proof}
Notice first that $\calp$ is not lower saturated because the objects 3 and 4 do have an upper bound, but do not have a lower bound whose vertex set is the intersection of their corresponding vertex sets.

In the calculation of functor cohomology in this example we have to consider chains of length 0, 1 and 2. However, since we wish to compute $\higherlim{\calp}{1}\, F$, and for every sequence the value of the $F$ is at the first object in the sequence, and since $F(x)$ is trivial for $x=*, 1, 2$, we only have to consider
\begin{itemize}
\item the objects $i = 3\ldots 6$ in dimension 0,
\item $i\to j$, $i=3,4$ and $j=5,6$ in dimension 1, and
\item none in dimension 2.
\end{itemize}

Let $\beta\in C^0(\calp, F)$. Then $\beta = (a,b,c,d)$, corresponding to the objects 3,4,5,6. Notice that the morphisms between these objects are all identities. Thus by definition, we have
\[\delta(a,b,c,d)  = (c-a, d-a, c-b, d-b),\]
where the target element is indexed by $3\to 5$, $3\to 6$, $4\to 5$, and $4\to 6$ respectively. The matrix of the differential in this dimension is
\[\left(\begin{matrix}-1&\phantom{-}0&\phantom{-}1&\phantom{-}0\\-1&\phantom{-}0&\phantom{-}0&\phantom{-}1\\\phantom{-}0&-1&\phantom{-}1&\phantom{-}0\\\phantom{-}0&-1&\phantom{-}0&\phantom{-}1\end{matrix}\right),\]
which is easily seen to be of rank 3. This shows that there is a nontrivial element in $\higherlim{\calp}{1}\, F$.
\end{proof}

Next we obtain an extension of the result by Notbohm and Ray \cite{NR}, that was restated  as \eqref{Eq-NR}.

\begin{Thm}
\label{Thm-chlgy_lower_sat}
  Let $\calp$ be a lower saturated poset, let $(\bXbA)=\{(X_v,A_v)\;|\;v\in V_\calp\}$ be a collection of pairs of spaces, and let $h^*$ be a generalised cohomology theory that satisfies the strong form of the K\"unneth formula. Let $\aaa=\{a_v\colon h^*(X_v)\to h^*(A_v)\;|\; v\in V_\calp\}$. Then
  $$h^*(\calz_\calp(\bXbA))\cong \calt_\calp(\aaa).$$
\end{Thm}

\begin{proof}
Since $\calz_\calp(\bXbA)$ is defined as a homotopy colimit, one has the Bousfield-Kan spectral sequence  \cite[XII.4.5]{BK}:
\begin{equation}
  \label{BKSS}
  E_2^{p,q}\cong H^p(\calp; h^{-q}\circ \calz^\calp_{\bXbA})\quad\Longrightarrow\quad h^{-p-q}(\calz_\calp(\bXbA)).
\end{equation}
By definition the functor $\calt_{\calp,\aaa}$ coincides with the composite $h^*\circ\calz^\calp_{\bXbA}$, and
by Theorem \ref{T vanishing},
\[E_2^{p,q} = \higherlim{\calp}{p}\calt_{\calp,\aaa} = 0\]
if $p>0$. Hence the spectral sequence collapses to its 0-th column, which by definition is isomorphic to $\calt_\calp(\aaa)$, and the claim follows.
\end{proof}

Without assuming the strong form of the K\"unneth formula, one still gets an expression for the cohomology of the polyhedral product of collections of the form $\{(X_v, *)\;|\; v\in V_\calp\}$, except in that case it is not given by the polyhedral tensor product. We end this section by considering this case.

\begin{Defi}\label{Def-retraction}
Let $\calp$ be a finite pointed poset, and let   $F\colon \calp\to\Top$ be a functor. We say that a functor $R\colon \calp\op\to\Top$ is a \hadgesh{retraction for $F$}  if the following conditions are satisfied:
\begin{enumerate}[\rm(a)]
  \item $R(x)=F(x)$ for all $x\in\calp$;
  \item $R(\iota_{x,y})\circ F(\iota_{x,y})=1_{F(x)}$ for all $x\le y\in \calp$.
\end{enumerate}

A retraction $R\colon \calp\op\to\Top$ for a functor $F\colon \calp\to\Top$ is \hadgesh{upper factoring} if for any $x,y\le z\in \calp$, there is $w\in\calp$ such that $w\le x,y$ and such that the diagram
$$\xymatrix{F(x)\ar[r]^{F(\iota_{x,z})}\ar[d]_{R(\iota_{w,x})}&F(z)\ar[d]^{R(\iota_{y,z})}\\
F(w)\ar[r]^{F(\iota_{w,y})}&F(y)}$$
commutes.
\end{Defi}

\begin{Lem}
  \label{upper lower}
  Let $\calp$ be a poset and let $F\colon \calp\to\Top$ be a functor. Let $R\colon \calp\op\to\Top$ be a retraction for $F$. Then for any functor $G\colon\Top\op\to\Ab$, the composite $G\circ R\op$ is a section for $G\circ F\op$. Moreover, if $R$ is upper factoring in addition, then $G\circ R\op$ is lower factoring.
\end{Lem}
\begin{proof}
From the definitions it is immediate that  composite $G\circ R\op\colon\calp \to \Ab$ is a section for  $G\circ F\op\colon \calp\op\to \Ab$  (see Section \ref{Sec Functors with a section} and Definition \ref{Def-retraction}). The second statement follows at once from Definition \ref{lower factoring}.
\end{proof}

Let $\calp$ be a poset and $\mathbf{X}=\{X_v\;|\;v\in V_\calp\}$ be a collection of pointed spaces. Define a functor $R^\calp_{\mathbf{X},*}\colon \calp\op\to\Top$ by $R^\calp_{X,*}=\calz_{\mathbf{X},*}^\calp$ on objects, and for $x\le y\in \calp$, let $R^\calp_{\mathbf{X},*}(\iota_{x,y})$ be the obvious projection. Then $R^\calp_{\mathbf{X},*}$ is  a retraction for $\calz^\calp_{\mathbf{X}, *}$. By analogy to Lemma \ref{lower saturated}, one has the following.

\begin{Lem}
  \label{lower saturated R}
  Let $\calp$ be a lower saturated poset and $\mathbf{X}=\{X_v\;|\;v\in V_\calp\}$ be a collection of pointed spaces. Then $R^\calp_{\mathbf{X},*}$ is an upper factoring retraction of $\calz^\calp_{\mathbf{X}, *}$.
\end{Lem}

\begin{Prop}
\label{Prop polyhedral product general}
  Let $\calp$ be a lower saturated poset, and let $\mathbf{X}=\{X_v\; | \; v\in V_\calp\}$ be a collection of pointed spaces indexed by the vertices of $\calp$.  Let $h^*$ be a generalised cohomology theory. Then
  $$h^*(\calz^\calp_{\mathbf{X}, *})\cong H^0(\calp, h^*\circ \calz^\calp_{\mathbf{X}, *}) = \lim_\calp (h^*\circ\calz^\calp_{\mathbf{X}, *})$$
\end{Prop}

\begin{proof}
By Lemmas \ref{upper lower} and \ref{lower saturated R}  the composite functor $h^*\circ \calz^\calp_{\mathbf{X}, *}$  has a lower factoring section. Hence by Theorem \ref{vanishing}, $H^n(\calp; h^*\circ \calz^\calp_{\mathbf{X}, *})=0$ for $n\ge 1$, and so the Bousfield-Kan spectral sequence for $h^*$ and $\calz^\calp_{\mathbf{X}, *}$ collapses onto its 0-column.The claim follows.
\end{proof}

Proposition \ref{Prop polyhedral product general} can be generalised to cover collections $\bXbA$ where for each $v\in V_\calp$, $A_v$ is a strict retract of $X_v$, i.e. when the inclusion $A_v\to X_v$ has a left inverse. The argument is essentially the same and is left to the reader.



\section{Polyhedral posets}
\label{Sec-Polyhedral}
In this section we specialise to a class of posets that are of primary interest in this article - \hadgesh{polyherdal posets}. Polyhedral posets enjoy a number of pleasant properties that we discuss above. In particular they are lower saturated, which makes them fit well with the analysis of previous sections.  In the main theorem of the section, Theorem \ref{Thm-PTP-relations}, we present a computation of $\calt_\calp(\mathbf{e})$ for a polyhedral poset $\calp$ and a collection of augmentation maps $e_v\colon P_\Bbbk[v]\to \Bbbk$, $v\in V_\calp$, in terms of generators and relations. This motivates a definition of the Stanley-Reisner ring of a polyhedral poset, which generalises the Stanley-Reisner ring of a simplicial poset (See Definition \ref{Def-SR-ring-Simplicial-Poset}).


\subsection{Definition and basic properties}

Recall that a \hadgesh{lower semilattice} is a poset $\calp$, such that any two objects $x,y\in\calp$ have a \hadgesh{meet}, i.e., a greatest lower bound $x\wedge y$. An important family of examples of lower semilattices occurs as  face posets of polyhedral complexes. Gr\"unbaum \cite[p.\;206]{G} refers to finite lower semilattices  as abstract complexes and so, by analogy to simplicial posets generalising the concept of the face poset of a simplicial complex, we present here \hadgesh{polyhedral posets} as a generalisation of lower semilattices.

\begin{Defi}
\label{Def-polyhedral_poset}
	A \hadgesh{polyhedral poset} is a pointed poset $\calp$ such that for any object  $x\in\calp$, the sub-poset $\calp_{\le x}$ is a lower semilattice.
\end{Defi}

\begin{Exmp}
	Simplicial posets are polyhedral posets. On the other hand the face poset of a cubical complex is polyhedral but not simplicial.
\end{Exmp}

The following characterisation of polyhedral posets is useful.

\begin{Prop}
\label{polyhedral upper bound}
  A pointed poset $\calp$ is a polyhedral poset if and only if any two objects $x,y\in\calp$ such that $[x\vee y]\nempty$, have a meet $x\wedge y$.
\end{Prop}

\begin{proof}
Assume $\calp$ is polyhedral. Let $x,y\in \calp$ be any two objects that admit a common upper bound $z$. Then any lower bound of $x$ and $y$ belongs to $\calp_{\le z}$, and since $\calp_{\le z}$ is a lower semilattice by definition, $x\wedge y$ exists in $\calp_{\le z}$ and thus in $\calp$.

Conversely, let  $z\in\calp$ be any object, and let  $x,y\in\calp_{\le z}$ be any two objects. Then the meet $x\wedge y$ exists in $\calp$ by assumption, and is an object of $\calp_{\le z}$ and a lower bound for $x$ and $y$ there.  Since $\calp_{\le z}$ is a sub-poset of $\calp$, maximality and uniqueness is clear, and so  $\calp_{\le z}$ is a lower semilattice.
\end{proof}

The following is an easy, yet important, property of polyhedral posets.
\begin{Prop}
  \label{polyhedral lower saturated}
  Polyhedral posets are lower saturated.
\end{Prop}
\begin{proof}
  Let $\mathcal{P}$ be a polyhedral poset, and let $x,y\in\calp$ be objects such that $[x\vee y]\neq \emptyset$. Then by Proposition \ref{polyhedral upper bound} $x\wedge y$ exists, and it suffices to show that  $V(x\wedge y)=V(x)\cap V(y)$. If $v\in V(x\wedge y)$, then $v\le x\wedge y\le x,y$, implying $v\in V(x)\cap V(y)$. On the other hand, if $u\in V(x)\cap V(y)$, then $u\le x,y$, and so $u\le x\wedge y$. Thus $u\in V(x\wedge y)$, and the claim follows.
\end{proof}


\subsection{Reduced posets}

A finite pointed poset $\calp$ is said to be \hadgesh{reduced} if $x\le y\in\calp$ and $V(x)=V(y)$ imply $x=y$. We show that any  finite pointed poset admits a quotient poset $\widehat{\calp}$ that is a reduced, and that reduction preserves  limits under a certain condition. We also show that the reduction of a polyhedral poset remains polyhedral. This is particularly useful when studying  polyhedral tensor products over polyhedral posets.

Let $\calp$ be a pointed poset and suppose that there are $*\ne x<y\in\calp$, such that there is no $z\in\calp$ with $x<z<y$. Define a new poset $\calp_{\setminus y}$ with object set  $\obj(\calp_{\setminus y})=\obj(\calp)\setminus \{y\}$, and $u\le v\in \calp_{\setminus y}$ if and only if $u<v$ in $\calp$ or $u<y$ and $v=x$ in $\calp$. Let  $\pi\colon\calp\to\calp_{\setminus y}$ be the poset map defined by
$$\pi(u)=\begin{cases}u&u\ne x,y\\x&u=x\; \text{or}\; y.\end{cases}$$
Note that $V_\calp=V_{\calp_{\setminus y}}$.

Let $F\colon \calp\op\to\Ab$ be a functor such that $F(x)=F(y)$ and $F(\iota_{x,y})$ is the identity map. Define a functor $F_{\setminus y}\colon\calp_{\setminus y}\op\to\Ab$ by
$$F_{\setminus y}(u)=F(u)\quad\text{and}\quad F_{\setminus y}(\iota_{u,v})=\begin{cases}F(\iota_{u,v})&u<v\text{ in }\calp\\F(\iota_{u,y})&u<y,\,x=v\text{ in }\calp,\end{cases}$$
One easily observes that the diagram
$$\xymatrix{
\calp^\mathrm{op}\ar[d]^{\pi^\mathrm{op}}\ar[rr]^F  && \Ab\ar@{=}[d]\\
\calp_{\setminus y}\op\ar[rr]^{F_{\setminus y}}&&\Ab}$$
commutes.

With this setup we can now prove the following useful lemma.

\begin{Lem}
	\label{reduction}
	Let $\calp$ be a pointed poset and suppose that there are $*\ne x<y\in\calp$ such that there is no $z\in\calp$ with $x<z<y$. Let $F\colon \calp\op\to\Ab$ be a functor such that $F(x)=F(y)$ and $F(\iota_{x,y})$ is the identity map.
	Then the map
	$$\pi^*\colon\lim_{\calp_{\setminus y}} F_{\setminus y}\longrightarrow\lim_\calp F$$
induced by the projection $\pi_{\setminus y}$ is an isomorphism.
\end{Lem}

\begin{proof}
The map
\[\pi^*\colon \prod _{p\in\calp_{\setminus y}}F_{\setminus y}(p) \to \prod_{p\in\calp}F(p) \]
induced by $\pi_{\setminus y}$ is the identity on all coordinates except the $y$-coordinate which it sends diagonally to the $x$ and $y$ coordinates in its target. Hence it is injective and it follows that $\pi_*$ is injective on the limit as well. On the other hand, if
\[a = \{a_p\;|\;p\in\calp\}\in \lim_\calp F\subseteq \prod_{p\in\calp}F(p)\]
is any element, then  for any morphism $\iota_{p,q}$ in $\calp$, $a_p = F(\iota_{p,q})(a_q)$. In particular $a_x = a_y$, and so $a$ is in the image of $\pi^*$.
\end{proof}

As an easy corollary we have the following.

\begin{Cor}
	\label{reduction tensor}
	Let $\calp$ be a finite pointed poset, and let  $\mathbf{a}=\{a_v\colon M_v\to N_v\;|\; v\in V_\calp\}$ be a collection of morphisms in $\Modk$. Assume that there are two objects $x<y\in \calp$ with $V(x) = V(y)$, and no $z$ such that $x<z<y$. Then there is an isomorphism
	$$\calt_\calp(\mathbf{a})\cong \calt_{\calp_{\setminus y}}(\mathbf{a}).$$
\end{Cor}
\begin{proof}
Since $V(x) = V(y)$, the map $\mathcal{T}_\calp(\mathbf{a})(\iota_{x,y})$ is the identity.
	By Lemma \ref{reduction}, there is an isomorphism $\calt_\calp(\mathbf{a})$
	$$\calt_\calp(\mathbf{a})=\lim \calt^\calp_{\mathbf{a}}\cong\lim\calt^{\calp_\setminus y}_{\mathbf{a}} = \calt_{\calp_\setminus y}(\mathbf{a}).$$
\end{proof}

\begin{Lem}
	\label{reduction lower semilattice}
	Let $\calp$ be a polyhedral poset, and assume that there are $x<y\in\calp$ such that $V(x)=V(y)$ and that there is no $z\in\calp$ with $x<z<y$. Then the poset $\calp_{\setminus y}$  also polyhedral.
\end{Lem}

\begin{proof}
Assume first that $\calp$ is a lower semilattice. By the assumptions on $x$ and $y$, for any $p\in\calp$, $p\wedge y = p\wedge x$. Fix objects $p,q\in\calp_{\setminus y}$. Then the meet $p\wedge q$ exists in $\calp$, because it  is a lower semilattice, and it is well defined there even in the case where $p$ or $q$ are $x$.  If $p\wedge q\neq y$ in $\calp$, then $\pi_{\setminus y}(p\wedge q)$ is obviously $p\wedge q$ in $\calp_{\setminus y}$. On the other hand, if $p\wedge q=y$, then by the definition of $\calp_{\setminus y}$, $p\wedge q=x$ there. This shows that $\calp_{\setminus y}$ is a lower semilattice.

If $\calp$ is a polyhedral poset, then $\calp_{\setminus y}$ can be obtained from $\calp$ by replacing $\calp_{\le y}$ by $(\calp_{\le y})_{\setminus y}$. Thus the the general claim follows from the previous argument.
\end{proof}

\begin{Cor}
\label{Cor-polyhedral_reduced}
	Let $\calp$ be a finite pointed poset and $\mathbf{a}=\{a_v\colon M_v\to N_v\;|\; v\in V_\calp\}$ be a collection of morphisms in $\Modk$. Then there is a reduced finite pointed poset $\widehat{\calp}$ such that $V_{\widehat{\calp}}=V_\calp$ and
	$$\calt_\calp(\mathbf{a})\cong\calt_{\widehat{\calp}}(\mathbf{a}).$$
	Furthermore, if $\calp$ is polyhedral, then so is $\widehat{\calp}$.
\end{Cor}

\begin{proof}
	Apply the reduction procedure to $\calp$ inductively. Since $\calp$ is finite this procedure terminates with a reduced  pointed poset $\widehat{\calp}$ with $V_\calp = V_{\widehat{\calp}}$, and by Corollary \ref{reduction tensor}
	$$\calt_\calp(\mathbf{a})\cong \calt_{\widehat{\calp}}(\mathbf{a}).$$
	In particular if $\calp$ is polyhedral, then by Lemma \ref{reduction lower semilattice} so is $\widehat{\calp}$.
\end{proof}


\subsection{Stanley-Reisner ring}

We now show that the definition of the Stanley-Reisner ring generalises naturally to  finite polyhedral posets.

\begin{Lem}
  \label{FPP-terminal}
 Let  $\calp$ be a finite pointed poset with a terminal object $x_0$. For each vertex $v\in V_\calp$ let $e_v\colon P_\Bbbk[v]\to \Bbbk$ be the augmentation map, and let
 $\mathbf{e}=\{e_v\colon P_\Bbbk[v]\to\Bbbk\;|\; v\in V_\calp\}$. Then the inclusion \[P_\Bbbk[V(x_0)] = P_\Bbbk[V_\calp]\to \prod_{x\in\calp}\calt_{\calp,\mathbf{e}}(x)\] induces an
 isomorphism
  $$P_\Bbbk[V_\calp]\xto{\cong}\calt_\calp(\mathbf{e})=\lim_\calp \calt_{\calp,\mathbf{e}}$$
and $\calt^n_\calp(\mathbf{e})=0$ for all $n>0$.
\end{Lem}
\begin{proof}
Since $\calp$ has a terminal object, $\calp\op$ on which $\calt_{\calp,\mathbf{e}}$ is defined has an initial object. Hence the limit of
$\calt_{\calp,\mathbf{e}}$ is its value on that initial object, and all higher limits vanish.
\end{proof}

Next we present a calculation of $\calt_\calp(\mathbf{e})$ for a finite polyhedral poset.

\begin{Thm}
\label{Thm-PTP-relations}
Let $\calp$ be a finite polyhedral poset. Then there is an isomorphism
\[\calt_\calp(\mathbf{e}) \cong P_\Bbbk[\obj(\calp)]/I_\calp,\]
  where the ideal $I_\calp$ is generated by:
  \begin{enumerate}[\rm (a)]
    \item $*-1$; \label{Thm-PTP-relations-a}
    \item $x-y$, if $x<y$ and $V(x)=V(y)$;\label{Thm-PTP-relations-b}
    \item $\prod_{x\in S}x$, for $S\subset \obj(\calp)$ with $[\vee S]=\emptyset$;\label{Thm-PTP-relations-c}
    \item
    $$\prod_{\substack{R\subset S\\|R|\text{ is odd}}}\bigwedge R-\left(\prod_{\substack{R\subset S\\|R|\text{ is even}}}\bigwedge R\right)\cdot\left(\sum_{\substack{z\in[\vee S]\\V(z)=\bigcup_{w\in S}V(w)}}z\right),$$
    for $S\subset \obj(\calp)$ with $[\vee S]\ne\emptyset$. \label{Thm-PTP-relations-d}
  \end{enumerate}
\end{Thm}

\begin{proof}
By Corollary \ref{Cor-polyhedral_reduced} we may assume that $\calp$ is reduced.
Assume first that $\calp$ is a lower semi lattice with a terminal object, so that $\calt_\calp(\mathbf{e})\cong P_\Bbbk[V_\calp]$, by Lemma \ref{FPP-terminal}. Consider the composite of the  inclusion followed by the projection
\[P_\Bbbk[V_\calp]\to P_\Bbbk[\obj(\calp)]\to P_\Bbbk[\obj(\calp)]/I_\calp.\]
 Notice first that relation  \ref{Thm-PTP-relations}(\ref{Thm-PTP-relations-a}) removes $*$ as a generator, while the relation \ref{Thm-PTP-relations}(\ref{Thm-PTP-relations-b}) is redundant since $\calp$ is assumed reduced.  Since $\calp$ has a terminal object, the relation (\ref{Thm-PTP-relations-c}) is also redundant. For any object $x\in\calp$  the relation (\ref{Thm-PTP-relations-d}) implies that   $x=\prod_{v\in V(x)}v$ in $P_\Bbbk[\obj(\calp)]/I_\calp$. Thus the composite above is an isomorphism, and by Lemma \ref{FPP-terminal} the theorem holds in this case.

By definition, for any finite pointed poset,
\[\calt_\calp(\mathbf{e}) = \lim_{\calp}\calt_{\calp,\mathbf{e}}\subseteq \prod_{x\in\calp}\calt_{\calp,\mathbf{e}}(x).\]
If $\calp$ is a finite polyhedral poset, then for each $x\in\calp$ the sub-poset $\calp_{\le x}$ is a lower semilattice with a terminal object. Hence $\calt_{\calp,\mathbf{e}}(x)$ is naturally isomorphic to $P_\Bbbk[\obj(\calp_{\le x})]/I_{\calp_{\le x}}$. For each $x\in\calp$, let
\[\pi_x\colon P_\Bbbk[\obj(\calp)]\to P_\Bbbk[\obj(\calp_{\le x})]/I_{\calp_{\le x}}\]
denote the map given by sending each $y\in\calp$ to the corresponding class in $P_\Bbbk[\obj(\calp_{\le x})]/I_{\calp_{\le x}}$ if $y\le x$, and to 0 otherwise. Assembling these maps together we obtain a map
\[\pi=\prod_{x\in\calp}\pi_x\colon P_\Bbbk[\obj(\calp)]\to \prod_{x\in\calp}P_\Bbbk[\obj(\calp_{\le x})]/I_{\calp_{\le x}}.\]

We claim that $\Ima(\pi) = \calt_\calp(\mathbf{e})\subseteq \prod_{x\in\calp}P_\Bbbk[\obj(\calp_{\le x})]/I_{\calp_{\le x}}$. Let $x_1^{k_1}x_2^{k_2}\cdots x_r^{k_r}$ be any monomial in the variables $x_i\in\calp$. Then for each $x\in\calp$,
\[\pi_x(x_1^{k_1}x_2^{k_2}\cdots x_r^{k_r}) = \left(\prod_{v_1\in V(x_1)}v_1^{k_1}\right)\left(\prod_{v_2\in V(x_2)}v_2^{k_2}\right)\cdots \left(\prod_{v_r\in V(x_r)}v_r^{k_r}\right),\]
if $x_i\le x,\;\forall i=1\ldots r$, and $\pi_x(x_1^{k_1}x_2^{k_2}\cdots x_r^{k_r}) = 0$ otherwise.
Here the image of  $x_i$ in $P_\Bbbk[\obj(\calp_{\le x})]/I_{\calp_{\le x}}$ is identifies with the product of its vertices. By definition, the image of each such monomial is a member of the inverse limit. Since $\pi$ is a homomorphism, its image is contained in the inverse limit. On the other hand, let $u = (u_x\;|\;x\in\calp)$ be an element in the inverse limit. Then for each $x\in\calp$ and all  $y\geq x$, $u_x=u_y$. It now follows easily that $u\in\Ima(\pi)$. This proves our claim.

It remains to show that $\Ker(\pi) = I_\calp$. The element $*-1$ is clearly in the kernel of each $\pi_x$. Since $\calp$ is reduced by assumption, the relation (\ref{Thm-PTP-relations-b}) is redundant. If $S\subseteq\obj(\calp)$ is such that $[\vee S]=\emptyset$, then by definition of the map $\pi$, the product $\prod_{x\in S} x\in\Ker(\pi)$. Hence the relation (\ref{Thm-PTP-relations-c}) holds. Finally let  $S\subseteq\obj(\calp)$ is such that $[\vee S]\neq\emptyset$. Then it suffices to show that for each $z\in[\vee S]$ the identity
\[\prod_{\substack{R\subset S\\|R|\text{ is odd}}}\bigwedge R-\left(\prod_{\substack{R\subset S\\|R|\text{ is even}}}\bigwedge R\right)\cdot z =0\]
holds in $P_\Bbbk[\obj(\calp_{\le z})]/I_{\calp_{\le z}}$. But, for each $x\le z$, one has the identity $x = \prod_{v\in V(x)} v$ in $P_\Bbbk[\obj(\calp_{\le z})]/I_{\calp_{\le z}}$.  Furthermore, if $R\subseteq S$, then in $P_\Bbbk[\obj(\calp_{\le z})]/I_{\calp_{\le z}}$ one has the identity
\[\wedge R = \prod_{v\in V(x)\atop\forall x\in R} v.\]
Hence we must show that the relation
\begin{equation}\label{IE-princ}
\prod_{\substack{R\subset S\\|R|\text{ is odd}}}\prod_{v\in V(x)\atop\forall x\in R} v = \left(\prod_{\substack{R'\subset S\\|R'|\text{ is even}}}\prod_{u\in V(y)\atop\forall y\in R'} u\right)\cdot \prod_{w\in V(z)} w
\end{equation}
holds in $P_\Bbbk[\obj(\calp_{\le z})]/I_{\calp_{\le z}}$. Fix a set of objects $S= \{x_1,\ldots x_m\}\subseteq\obj(\calp)$, and for each $1\le i\le m$, let $A_i = V(x_i)$. Then the relation (\ref{IE-princ}) can be rewritten as
\[\prod_{\substack{\emptyset\ne I\subset\{1,\ldots,m\}\\|I|\text{ is odd}}}\prod_{v\in \bigcap_{i\in I}A_i} v = \left(\prod_{\substack{\emptyset\ne J\subset\{1,\ldots,m\}\\|J|\text{ is even}}}\prod_{u\in \bigcap_{j\in J}A_j} u \right)\cdot \prod_{w\in \bigcup_{k=1}^m A_k} w,\]
and this follows at once from the
 inclusion-exclusion principle
$$\left|\bigcup_{i=1}^mA_i\right|=\sum_{\emptyset\ne I\subset\{1,\ldots,m\}}(-1)^{|I|-1}\left|\bigcap_{i\in I}A_i\right|$$
for any family of finite sets $A_1,\ldots,A_m$. This shows that $I_\calp\subseteq \Ker(\pi)$.

Conversely, let $q\in\Ker(\pi)$ be any element. Then $q$ can be expressed uniquely as a linear combination $q = a_1m_1+a_2m_2+\cdots a_rm_r$, where $m_i$ are monomials on the objects of $\calp$ and $a_i\in\Bbbk$. Without loss of generality we may assume that no $m_i$ is divisible by a product  of objects without an upper bound, since those monomials would be  in $I_\calp$ by definition. Furthermore, $\pi(q)=0$ if and only if $\pi_z(q)=0$ for every maximal object $z\in\calp$, by definition of the inverse limit. Hence we may restrict attention only to those maximal objects $z$, such that for each $1\le i\le r$, every variable present in the monomial $m_i$ is bounded above by $z$. Thus, for each such maximal object $z$, $q\in P_\Bbbk[\obj(\calp_{\le z})]$. Now, since the composite
\[P_\Bbbk[\obj(\calp_{\le z})]\xto{\inc}P_\Bbbk[\obj(\calp)]\xto{\proj}P_\Bbbk[\obj(\calp_{\le z})]\] is the identity, and since $q$ may be assumed to be an element of $P_\Bbbk[\obj(\calp_{\le z})]$, it follows that $q$ is in the kernel of the composite
\[P_\Bbbk[\obj(\calp_{\le z})]\xto{\inc}P_\Bbbk[\obj(\calp)]\xto{\proj}P_\Bbbk[\obj(\calp_{\le z})]/I_{\calp_{\le z}}\]
which is $I_{\calp_{\le z}}\subseteq I_\calp$. Thus $\Ker(\pi)\subseteq I_\calp$, and so equality holds. This completes the proof of the theorem.
\end{proof}

\begin{Cor}\label{Cor-generalization-of-simp}
Let $\calp$ be a simplicial poset, and let $\Bbbk[\calp]$ be its  Stanley-Reisner ring over $\Bbbk$. Then there is an isomorphism of $\Bbbk$-algebras
\[\Bbbk[\calp]\cong\calt_\calp(\mathbf{e}).\]
\end{Cor}
\begin{proof}
This follows at once from Theorem \ref{Thm-PTP-relations}, since a simplicial poset is in particular polyhedral.
\end{proof}

Corollary \ref{Cor-generalization-of-simp} provides the motivation to generalise the definition of the Stanley-Reisner ring over $\Bbbk$ to finite polyhedral posets.

\begin{Defi}
\label{Def-general-SR-ring}
  Let $\calp$ be a finite polyhedral poset. Define the Stanley-Reisner ring of $\calp$ over $\Bbbk$  to be the quotient ring
  $$\Bbbk[\calp]\defeq P_\Bbbk[\obj(\calp)]/I_\calp,$$
  where $I_\calp$ is  the ideal  defined in Theorem \ref{Thm-PTP-relations}
  \end{Defi}

As an easy consequence we obtain a generalisation of the  Davis-Januszkiewicz result \cite[Theorem 4.8]{DJ} (See Equation \eqref{k[K] DJ}).

\begin{Cor}
  \label{DJ k[P]}
Let $\calp$ be a finite polyhedral poset. Then  there is an isomorphism
  $$H^*(\calz_\calp(\C P^\infty,*),\Bbbk)\cong\Bbbk[\calp].$$
\end{Cor}
\begin{proof}
By Proposition \ref{polyhedral lower saturated}, $\calp$ is lower saturated, and by Theorem \ref{Thm-chlgy_lower_sat}, $H^*(\calz_\calp(\C P^\infty,*))$ is isomorphic to the polyhedral tensor product $\calt_\calp(\mathbf{e})$, where $\mathbf{e}$ is the collection of augmentations $\{P_\Bbbk[v]\to\Bbbk\;|\; v\in V_\calp\}$. The claim now follows from Theorem \ref{Thm-PTP-relations}.
\end{proof}



\section{Colimits and homotopy colimits}
\label{Sec-Colim_vs_Hocolim}

In this section we discuss a comparison of colimits and homotopy colimits. These are of course not the same in general, but do coincide up to homotopy under certain circumstances. In particular, in our context, polyhedral products for simplicial complexes are  defined as colimits, whereas we defined them as homotopy colimits. Hence our aim is to show that in familiar cases our definition coincides up to homotopy with the standard definition.

We start with a useful concept introduced by Ziegler and \u{Z}ivaljevi\'c \cite{ZZ}.
\begin{Defi}\label{Def-arrangement}
Let $X$ be a topological space and $\cala=\{A_1,\ldots,A_r\}$ be a collection of subspaces of $X$. The collection $\mathcal{A}$ is said to be an \hadgesh{arrangement}, if the following conditions hold:
\begin{enumerate}[\rm (a)]
	\item For any $A,B\in\mathcal{A}$, $A\cap B$ is a union of elements of $\mathcal{A}$, and \label{arrangement-a}
	\item For any $A,B\in\mathcal{A}$ such that $A\subset B$, the inclusion $A\to B$ is a cofibration. \label{arrangement-b}
\end{enumerate}
\end{Defi}
The elements of an arrangement $\mathcal{A}$  admits a partial order given by inclusion. The resulting poset is referred to as  the intersection poset of $\mathcal{A}$  (which differs from the usual sense of intersection posets becausethe usual one consists of all intersections of elements of $\cala$). The following is proved in \cite[Projection Lemma 1.6]{ZZ}.

\begin{Lem}
	\label{projection lem}
	Let $\cala$ be an arrangement with the intersection poset $I(\cala)$, and let $F\colon I(\cala)\to\Top$ be the  functor that associates with a member of $I(\cala)$ the corresponding element of $\cala$. Then the natural map
	$$\hocolim_{I(\cala)}\,F\to\colim_{I(\cala)}\,F$$
	is a homotopy equivalence.
\end{Lem}

The following lemma is immediate from the definitions.

\begin{Lem}
	\label{support}
	Let $\calp$ be a finite pointed poset and $(\bXbA)=\{(X_v,A_v)\;|\; v\in V_\calp\}$ be a collection of pairs of spaces. For an object $x\in\calp$ and a point $\mathbf{a}=(a_w)_{w\in V_\calp}\in\calz^\calp_{\bXbA}(x)$, define the \hadgesh{support} of $a$ by
$$\mathrm{supp}(\mathbf{a})=\{v\in V_\calp\,\vert\, a_v\notin A_v\}.$$
Then the following statements hold:
	\begin{enumerate}[\rm (a)]
		\item $\mathrm{supp}(\mathbf{a})\subset V(x)$, and
		\item for any $x\le y\in\calp$, $\mathrm{supp}(\calz^\calp_{\bXbA}(\iota_{x,y})(\mathbf{a}))=\mathrm{supp}(\mathbf{a})$.
	\end{enumerate}
\end{Lem}

\begin{Lem}
\label{Lem-homeo_into}
Let $\calp$ be a finite polyhedral poset and $(\bXbA)=\{(X_v,A_v)\;|\;v\in V_\calp\}$ be a collection of pairs of spaces. For each $x\in \calp$ let $Z_x$ denote the image of the natural map $j_x\colon\calz^\calp_{\bXbA}(x)\to\mathrm{colim}\,\calz^\calp_{\bXbA}$. Then the following statements hold:
\begin{enumerate}[\rm (a)]
\item If $\mathbf{a}\in\calz^\calp_{\bXbA}(x)$ and $\mathbf{b}\in\calz^\calp_{\bXbA}(y)$ are points such that $j_x(\mathbf{a}) = j_y(\mathbf{b})$, then there is an object $u\in\calp$ and $\mathbf{d}\in \calz^\calp_{\bXbA}(u)$, such that $u\le x,y$, and $\mathbf{a}=\calz^\calp_{\bXbA}(\iota_{u,x})(\mathbf{d})$ and $\mathbf{b}=\calz^\calp_{\bXbA}(\iota_{u,y})(\mathbf{d})$. \label{lower-bound-image}

\item For each object $x\in \calp$, the map $j_x$ is injective. \label{injective}
\end{enumerate}
\end{Lem}
\begin{proof}
Part (\ref{injective}) follows from (\ref{lower-bound-image}) upon setting $x=y$. Thus it suffices to prove  (\ref{lower-bound-image}).

Assume that $\mathbf{a}\in\calz^\calp_{\bXbA}(x)$ and $\mathbf{b}\in\calz^\calp_{\bXbA}(y)$ are points such that $j_x(\mathbf{a}) = j_y(\mathbf{b})$. Then there is a zig-zag of morphisms in $\calp$,
\[\xymatrix{
& x_1 && x_2 & \cdots &\cdots & \cdots & x_k\\
y_0=x\ar[ur] && y_1\ar[ul]\ar[ur] && y_2\ar[ul] & \cdots & y_{k-1}\ar[ur] && y_k=y\ar[ul]
}\]
Such that upon applying the functor $\calz^\calp_{\bXbA}$ and mapping the resulting diagram to the colimit via the maps $j_{(-)}$, we obtain  a corresponding zig-zag of elements and morphisms connecting them, with $j_x(\mathbf{a})$ and $j_y(\mathbf{b})$
at the corresponding ends of the diagram. Hence there exists a meet $y_0\wedge y_1$, and by induction one obtains a meet $z=y_0\wedge\cdots\wedge y_k$. By Lemma \ref{support},
\[\mathrm{supp}(\mathbf{a})=\mathrm{supp}(\mathbf{b})\subset V(z)=\bigcap_{i=0}^kV(y_i).\]
Hence there is $\mathbf{c}\in\Z_{\bXbA}^\calp(z)$ such that $\mathbf{a}=\calz^\calp_{\bXbA}(\iota_{z,x})(\mathbf{c})$ and $\mathbf{b}=\calz^\calp_{\bXbA}(\iota_{z,y})(\mathbf{c})$, completing the proof.

\end{proof}

Recall that a poset $\calp$ is reduced if $x\le y\in\calp$ and $V(x)=V(y)$ imply $x=y$.

\begin{Lem}
	\label{Lem-arrangement}
	Let $\calp$ be a finite reduced polyhedral poset and $(\bXbA)=\{(X_v,A_v)\;|\;v\in V_\calp\}$ be a collection of NDR-pairs of spaces. Assume that $X_v\ne A_v$ for all $v\in V_\calp$.
Then the collection $\{Z_x\}_{x\in\calp}$ is an arrangement (see Definition \ref{Def-arrangement}), whose intersection poset $I(\{Z_x\}_{x\in\calp})$ is isomorphic to $\calp$.
\end{Lem}
\begin{proof}
Let $x,y\in \calp$ be any  objects, and let $Z_x, Z_y$ be the corresponding images, as in Lemma \ref{Lem-homeo_into}. If $Z_x\cap Z_y\nempty$, then there are points  $\mathbf{a}\in\calz^\calp_{\bXbA}(x)$ and $\mathbf{b}\in\calz^\calp_{\bXbA}(y)$ such that $j_x(\mathbf{a}) = j_y(\mathbf{b})$. Hence by Lemma  \ref{Lem-homeo_into}(\ref{lower-bound-image}) there is an object $u\in\calp$ and $\mathbf{d}\in \calz^\calp_{\bXbA}(u)$, such that $\mathbf{a}=\calz^\calp_{\bXbA}(\iota_{u,x})(\mathbf{d})$ and $\mathbf{b}=\calz^\calp_{\bXbA}(\iota_{u,y})(\mathbf{d})$. Thus
	\begin{equation}
		\label{Z_x}
		Z_x\cap Z_y=\bigcup_{w\le x,y}Z_w,
	\end{equation}
	which is Condition (\ref{arrangement-a})  in Definition \ref{Def-arrangement}.

Since we are assuming that all $(X_v,A_v)$ are NDR-pairs, Condition (\ref{arrangement-b})  in Definition \ref{Def-arrangement} is also satisfied.  This shows that the collection $\{Z_x\}_{x\in\calp}$ is an arrangement.

Finally, we show that the intersection poset of the collection is isomorphic as a category to $\calp$. To do so we must show that all $Z_x$ are distinct subspaces in the colimit. By Lemma \ref{Lem-homeo_into}(\ref{injective}), for each object $x\in\calp$, the map $j_x$ is injective, and by assumption $X_v\neq A_v$ for all $v\in V_\calp$. Furthermore, by assumption $\calp$ is reduced, and so comparable objects cannot have the same vertex set. Hence all $Z_x$ are distinct and the natural functor  $J\colon\calp\to I(\{Z_x\}_{x\in\calp})$, which takes $x$ to $Z_x$ is an isomorphism of categories.
\end{proof}

\begin{Prop}
	\label{hocolim-colim}
	Let $\calp$ be a finite reduced polyhedral poset and $(\bXbA)=\{(X_v,A_v)\;|\;v\in V_\calp\}$ be a collection of NDR-pairs of spaces. If $X_v\ne A_v$ for all $v\in V_\calp$, then the natural map
	$$\hocolim_\calp\,\calz^\calp_{\bXbA}\to\colim_\calp\,\calz^\calp_{\bXbA}$$
	is a homotopy equivalence.
\end{Prop}
\begin{proof}
By Lemma \ref{Lem-arrangement}, under our assumptions, the functor $J\colon\calp\to I(\{Z_x\}_{x\in\calp})$ is an isomorphism of categories. Let
\[F\colon  I(\{Z_x\}_{x\in\calp})\to \Top\]
be the functor which takes object $Z_x$ to the underlying topological space.  Then there is a natural transformation
$j\colon \calz_{\bXbA}^\calp\to F\circ J$ that takes an object $x\in\calp$ to the continuous bijection $j_x\colon\calz_{\bXbA}^\calp(x)\to Z_x$ (see Lemma \ref{Lem-homeo_into}(\ref{injective})). By  naturality of the map from the homotopy colimit to the colimit, one has the following commutative diagram:

\[\xymatrix{
\hocolim_\calp\calz_{\bXbA}^\calp\ar[r]_\simeq^{j_*}\ar[d] & \hocolim_\calp F\circ J\ar[r]_\simeq^{J^*}\ar[d] & \hocolim_{I(\{Z_x\}_{x\in\calp})} F\ar[d]^\simeq\\
\colim_\calp\calz_{\bXbA}^\calp\ar[r]^{j_*}_\simeq &\colim_\calp F\circ J\ar[r]^{J^*} & \colim_{I(\{Z_x\}_{x\in\calp})}
}\]
The maps $j_*$ in the top and bottom  rows are homotopy equivalences because $j_x$ is an equivalence for each object $x\in\calp$. The maps $J^*$ are homotopy equivalences because $J$ is an equivalence of categories by Lemma \ref{Lem-arrangement}.
The right vertical map is a homotopy equivalence by the Projection Lemma  \ref{projection lem}. Thus the left vertical map is a homotopy equivalence as claimed.
\end{proof}



\section{The simplicial transform}
\label{Sec-Simplicial_Transform}

In this section we construct for any finite pointed poset $\calp$, a simplicial poset $s(\calp)$ with the same vertex set $V_\calp$, and a map of posets  $\calp\to s(\calp)$ that is an embedding if $\calp$ is reduced and  polyhedral. The main theorem of the section is the statement that for any finite reduced polyhedral poset $\calp$ and any collection of pairs $(\bXbA)$ indexed by the $V_\calp$, the polyhedral product of $(\bXbA)$ over $\calp$ coincides with its polyhedral product over $s(\calp)$.

\begin{Defi}\label{Def-Simplicial_transform}
Let $\calp$ be a pointed finite poset. Define an equivalence relation on the set
\[\{(x,S)\in \obj(\calp)\times 2^{V_\calp}\,\vert\,x\in\calp,\;S\subseteq V(x)\}\]
to be the transitive closure of the relation $(x,S)\sim(y,T)$, if $S=T$ and $[x\vee y]\nempty$. Let  $s(\calp)$ denote the set of equivalence classes of this relation, and define a poset structure on $s(\calp)$ by setting $[x,S]\le[y,T]$ if $S\subset T$ and $[x\vee y]\nempty$. We refer to the poset $s(\calp)$ as the \hadgesh{simplicial transform of $\calp$}.
\end{Defi}

\begin{Rma}
\label{Ref-Obj_rep}
Notice that if $x\le y\in\calp$ and $S\subseteq V(x)$, then $[x,S]=[y,S]$. Thus every object $[x,S]\in s(\calp)$ can be represented by $[y,S]$, where $y$ is any maximal  object in $\calp$ such that $y\geq x$. On the other hand, if $y', y''\in\calp$ are two maximal objects such that $S\subseteq V(y'), V(y'')$, but there is no lower bound $x\in[y'\wedge y'']$ such that $S\subseteq V(x)$, then $[y', S]$ and $[y'', S]$ are distinct objects of $s(\calp)$.
\end{Rma}

A justification for the term ``simplicial transform'' is given in the following lemma.
\begin{Lem}
	\label{s(P) simplicial}
	For any finite pointed poset $\calp$, its simplicial transform $s(\calp)$ is a simplicial poset with the same vertex set as $\calp$.
\end{Lem}
\begin{proof}
By Definition \ref{Def-Simplicial_transform},  for each object $[x,S]\in s(\calp)$, the sub-poset $s(\calp)_{\le[x,S]}$ consists of all objects $[y,T]$, such that $T\subset S$ and $[x\vee y]\nempty$. Thus in particular for each $[y,T]\in s(\calp)_{\le[x,S]}$, $[y,T] = [x,T]$, so  $s(\calp)_{\le[x,S]}$ is isomorphic to the boolean algebra $2^S$. This shows that $s(\calp)$ is a simplicial poset. The second statement is obvious.
\end{proof}

For a finite pointed poset $\calp$, one has a poset map
\begin{equation}
s\colon\calp\to s(\calp),\quad x\mapsto[x,V(x)].\label{Eq-Simplicial_embeding}
\end{equation}

\begin{Prop}
\label{Prop-Simp_emb}
	If $\calp$ is a finite reduced polyhedral poset, then the map $s\colon\calp\to s(\calp)$ is an embedding of $\calp$ in $s(\calp)$, with image the sub-poset of $s(\calp)$ generated by all objects $[x, V(x)]$, $x\in\calp$.
\end{Prop}

\begin{proof}
	Let $x,y\in\calp$ be any objects, and suppose
	\[[x, V(x)] = s(x) =s(y) = [y,V(y)].\]
	 Then there is a zig-zag $x\le z_1\ge x_1\le\cdots\le z_n\ge x_n=y$ such that $V(x)=V(x_1)=\cdots=V(x_n)=V(y)$. Since $\calp$ is a polyhedral poset, there is $x\wedge x_1$ such that $V(x\wedge x_1)=V(x)\cap V(x_1)=V(x)=V(x_1)$. Then since $\calp$ is reduced, $x=x_1=x\wedge x_1$. Thus by induction, one gets $x=y$, and so $s$ is injective.

	Suppose that $[x,V(x)]\le[y,V(y)]$ for $x,y\in\calp$. Then there are zig-zags $x\le z_1\ge x_1\le\cdots\le z_n\ge x_n=y$ and $V(x)\subset S_1\subset\cdots\subset S_n=V(y)$, where $S_i\subset V(x_i)$ for all $i$. Since $\calp$ is a polyhedral poset, the meet $x\wedge x_1$ exists, and  $V(x\wedge x_1)=V(x)\cap V(x_1)=V(x)$, and since $\calp$ is reduced, it follows that  $x=x\wedge x_1\le x_1$. By induction, $x\le x_n=y$, and so $s$ is an embedding as stated.
\end{proof}

Let $\calp$ be a finite pointed poset and $x\in\calp$ be a maximal element. Let $\calp_{\setminus x}\subseteq\calp$ denote the sub-poset on all objects but $x$. Then $\calp = \calp_{\setminus x}\cup \calp_{\le x}$, with $ \calp_{\setminus x}\cap \calp_{\le x} = \calp_{< x}$. Hence the following is a pushout diagram in the category of posets and poset maps:
\begin{equation}
	\label{pushout P}
	\xymatrix{\calp_{<x}\ar[r]\ar[d]&\calp_{\le x}\ar[d]\\
	\calp_{\setminus x}\ar[r]&\calp.}
\end{equation}
Next, we observe that this relation is preserved under simplicial transforms.

\begin{Lem}
	Let $\calp$ be a finite pointed poset and $x\in\calp$ be a maximal element. Then there is a pushout diagram
	\begin{equation}
		\label{pushout SP}
		\xymatrix{s(\calp_{<x})\ar[r]\ar[d]&s(\calp_{\le x})\ar[d]\\
		s(\calp_{\setminus x})\ar[r]&s(\calp).}
	\end{equation}
\end{Lem}

\begin{proof}
	By definition  it follows easily that for each object $y\in\calp$, $s(\calp_{\le y})=s(\calp)_{\le[y,V(y)]}$. Thus, restricting attention to object sets,
	$$s(\calp_{\setminus x})\cap s(\calp_{\le x})=\bigcup_{y<x}s(\calp_{\le y})=\bigcup_{y<x}s(\calp)_{\le[y,V(y)]}=s(\bigcup_{y<x}\calp_{\le y})=s(\calp_{<x}).$$
 If $a\le b$ in $s(\calp)$, then $a\le b$ in either $s(\calp_{\le x})=s(\calp)_{[x,V(x)]}$ or $s(\calp_{\setminus x})$. The claim follows.
\end{proof}

Let $\calp$ be a finite pointed poset. By construction, $V_\calp=V_{s(\calp)}$ and $V([x,S])=S$ for $[x,S]\in s(\calp)$. Then there is a commutative diagram
$$\xymatrix{\calp\ar[r]^{\calz_{\bXbA}^\calp}\ar[d]^s&\mathbf{Top}\ar@{=}[d]\\
s(\calp)\ar[r]^{\calz_{\bXbA}^{s(\calp)}}&\mathbf{Top}}$$
for any collection of pairs of spaces $(\bXbA)=\{(X_v,A_v)\;|\; v\in V_\calp\}$. Then in particular, there is a natural map
\begin{equation}
	\label{P->s(P)}
	\calz_\calp(\bXbA)\to\calz_{s(\calp)}(\bXbA).
\end{equation}

\begin{Thm}
	\label{projection lemma}
	Let $\calp$ be a finite reduced polyhedral poset and $(\bXbA)=\{(X_v,A_v)\;|\; v\in V_\calp\}$ be a collection of NDR-pairs such that $X_v\ne A_v$ for all $v$. Then the natural map \eqref{P->s(P)} is a homotopy equivalence.
\end{Thm}

\begin{proof}
	By Proposition \ref{hocolim-colim} it suffices to show that natural map
	\begin{equation}
		\label{colim P->s(P)}
		s^*\colon \mathrm{colim}\,\calz^\calp_{\bXbA}\to\mathrm{colim}\,\calz^{s(\calp)}_{\bXbA}
	\end{equation}
	is a homeomorphism. For each $x\in\calp$, let  $Z_x$ denote the subspace $\calz_{\bXbA}^\calp(x)\subseteq\mathrm{colim}\,\calz_{\bXbA}^\calp$. Since $\calp$ is a polyhedral poset, it follows from \eqref{Z_x} that for any maximal $x\in\calp$,
	$$\colim\,\calz^{\calp_{\setminus x}}_{\bXbA}=\bigcup_{x\ne y\in\calp}Z_y,\quad\colim\,\calz^{\calp_{\le x}}_{\bXbA}=Z_x,\quad\colim\,\calz^{\calp_{<x}}_{\bXbA}=\bigcup_{x>y\in\calp}Z_y.$$
	Thus one gets a pushout diagram of spaces
	\begin{equation}\label{Eq-pushout_P}\xymatrix{\mathrm{colim}\,\calz^{\calp_{<x}}_{\bXbA}\ar[r]\ar[d]&\mathrm{colim}\,\calz^{\calp_{\le x}}_{\bXbA}\ar[d]\\
	\colim\,\calz^{\calp_{\setminus x}}_{\bXbA}\ar[r]&\mathrm{colim}\,\calz^\calp_{\bXbA}.}\end{equation}
	Similarly, by Lemma \ref{s(P) simplicial} and \eqref{pushout SP}, one gets a pushout
	\begin{equation}\label{Eq-pushout_sP}\xymatrix{\mathrm{colim}\,\calz^{s(\calp_{<x})}_{\bXbA}\ar[r]\ar[d]&\mathrm{colim}\,\calz^{s(\calp_{\le x})}_{\bXbA}\ar[d]\\
	\mathrm{colim}\,\calz^{s(\calp_{\setminus x})}_{\bXbA}\ar[r]&\mathrm{colim}\,\calz^{s(\calp)}_{\bXbA}}\end{equation}
	and the natural maps \eqref{colim P->s(P)} give a map of diagrams from (\ref{Eq-pushout_P}) to (\ref{Eq-pushout_sP}).

	If all objects of $\calp$ are vertices, then $\calp = s(\calp)$, and \eqref{colim P->s(P)} is the identity map. Also if $\calp$ has a terminal object $x$, then the claim is obvious. Assume that for any maximal $y\in\calp$ one has $|V(y)|\le n$. Assume by induction that   \eqref{colim P->s(P)} is a homeomorphism for any poset $\calp$ in which no object has more than $n$ vertices, and the number of maximal objects $y$ of $\calp$ with $|V(y)|=n$ is less than $k$. Let  $\calp$ be a poset with $k$ maximal objects $y$ such that $|V(y)| = n$. Let $x\in\calp$ be any maximal object with this property. Then by the induction hypothesis, the maps $s^*$ from the top left and the bottom left corners of Diagram (\ref{Eq-pushout_P}) to the corresponding corners of Diagram (\ref{Eq-pushout_sP}) are homeomorphisms. The  map $s^*$ from the top right corner of (\ref{Eq-pushout_P}) to the top right corner of (\ref{Eq-pushout_sP}) is the clearly the identity map, since the posets in question have a terminal object. Hence $s^*$ on the pushout spaces
\[s^*\colon \colim_{\calp}\calz_{\bXbA}^\calp\to \colim_{s(\calp)}\calz_{\bXbA}^{s(\calp)}\]
	 is a homeomorphism as claimed.
\end{proof}

\begin{Cor}
\label{Cor-SRposet=SRst-poset}
	For a finite polyhedral poset $\calp$, there is an isomorphism
	$$\Bbbk[\calp]\cong\Bbbk[s(\calp)].$$
\end{Cor}

\begin{proof}
	This follows at once from Corollary \ref{DJ k[P]} and Theorem \ref{projection lemma}.
\end{proof}



\section{$f$-vectors of regular polyhedral posets}
\label{Sec-Regular_Polyhedral_Posets}

Let $\calp$ be a finite pointed poset. Define $||\calp||=\max\{|V_\calp(x)|\,\vert\,x\in\calp\}-1$, and for each $i\geq 0$, let \hadgesh{$f_i(\calp)$}  be the number of objects  $x\in\calp$ such that $|V_\calp(x)|=i+1$ for $i\ge 0$. Then in particular, $f_{-1}(\calp)=1$. We call the sequence $(f_0(\calp),\ldots,f_{|\calp|}(\calp))$ the \hadgesh{$f$-vector of $\calp$}. In this section, we compare the $f$-vectors of  finite polyhedral posets satisfying a certain regularity condition and their simplicial transformation $s(\calp)$. As a result we obtain a computation of the Poincar\'e series of the Stanley-Reisner ring of a finite polyhedral poset satisfying this condition in terms of the $f$-vector.

\begin{Defi}
	A finite pointed poset $\calp$ is said to be \hadgesh{regular} if $\calp_{\le x}\cong\calp_{\le y}$ for any objects $x, y\in\calp$ such that $|V(x)|=|V(y)|$.
\end{Defi}

A finite regular poset is automatically reduced. Notice also that in any finite pointed poset $\calp$, and any $x\in\calp$, the objects of the simplicial transform $s(\calp_{\le x})$ can be written in the form $[x, S]$, where $S\subseteq V(x)$.

\begin{Prop}
	\label{Prop-f(s(P))}
	Let $\calp$ be a finite regular polyhedral poset, and let $k$ be a nonnegative integer. Fix an object $x\in\calp$ such that $|V(x)|=k+1$.
	Set $\nu_{0,k}(\calp) = 0$, and for $1\le i\lneq k$ let \hadgesh{$\nu_{i,k}(\calp)$ }be the number of all objects $[x,S]\in \obj(s(\calp_{\le x}))$, such that
	\begin{enumerate}[\rm i)]
	\item $|S|=i+1$, and
	\item there is no $y\lneq x$ such that $S\subseteq V(y)$.
	\end{enumerate}
	Then for $i\ge 0$,
	$$f_i(s(\calp))=f_i(\calp)+\sum_{k=i+1}^{|\calp|}f_k(\calp)\nu_{i,k}(\calp).$$
\end{Prop}

\begin{proof}
Notice first that $\nu_{i,k}(\calp)$ is well defined, since the isomorphism type of $\calp_{\le x}$ depends only on $|V(x)|$ and not on a specific choice of the object $x$. Also the statement clearly holds for $i=0$, so we may assume $i\geq 1$.

By Definition $f_i(s(\calp))$ is the number of objects $[x,S]\in s(\calp)$ with $|S| = i+1$. Clearly, for each $x\in\calp$ such that $|V(x)| = i+1$, the object $[x, V(x)]\in s(\calp)$ contributes 1 to the count $f_i(s(\calp))$.  By Proposition \ref{Prop-Simp_emb} the functor $s\colon \calp\to s(\calp)$ is an embedding, and so objects of this form contribute $f_i(\calp)$ to $f_i(s(\calp))$.

The remaining objects have the form $[x,S]$, where $S\subsetneq V(x)$, $|S|=i+1$, and there is no $y\le x'\in\calp$ such that $S\subseteq V(x')$, $[x\vee x']\nempty$ and $S=V(y)$ (or else $[x,S]=[x',S] = [y, S]\in\Ima(s)$; see Remark \ref{Ref-Obj_rep}).  But if such $x'$ that is not comparable to $x$,  and $y\leq x'$ that satisfy these conditions exist, then for any $z\in[x\vee x']$ the sub-poset  $\calp_{\le z}$ is not a lower semilattice. To see this, assume the converse. Then $S$ consists of at least 2 of vertices common to both $x$ and $x'$. By condition ii), $x$ is a minimal upper bound for $S$ in $\calp_{\le x}$. Let $a\in \calp_{\le x'}$  be a minimal upper bound for $S$ in  $\calp_{\le x'}$. Then $a\neq x$, and every $v\in S$ is a maximal lower bound for $x$ and $a$ in $\calp_{\le z}$.

If $x\in\calp$ is such that $|V(x)|=k+1$, then the number of objects of the form $[x,S]\in \obj(s(\calp))\setminus \Ima(s)$ that are not equivalent to some $[y, S]$ for any $y$ with $|V(y)|\leq k$ is exactly $\nu_{i,k}(\calp)$. Hence the remaining objects are counted by the sum $\sum_{k\geq  i+1} f_k(\calp)\nu_{i,k}(\calp)$. This completes the proof.
\end{proof}

For a fixed  finite regular polyhedral poset $\calp$ and a positive integer $k$, let \hadgesh{$\calp[k]$} denote a poset that is isomorphic to $\calp_{\le x}$ for any $x\in\calp$ with $|V(x)| = k+1$.  We obtain an immediate corollary.

\begin{Cor}
\label{Cor-nuP=nuP[k]}
Let $\calp$ be a finite regular polyhedral poset. Then For every $0\le i\lneq k\le n\le |\calp|$, $\nu_{i,k}(\calp) = \nu_{i,k}(\calp[n])$.
\end{Cor}

Proposition \ref{Prop-f(s(P))} allows us to obtain a recursive formula for $\nu_{i,k}(\calp)$ from the $f$-coordinates  $f_j(\calp[k])$.

\begin{Prop}
\label{Prop-nu_ik}
Let $\calp$ be a finite regular polyhedral poset. Then
	 for each $0\le i \lneq n$,
\[\nu_{i,n}(\calp) = {n+1\choose i+1} -\left(f_i(\calp[n]) + \sum_{k=i+1}^{n-1}f_k(\calp[n])\nu_{i,k}(\calp) \right).\]
\end{Prop}
\begin{proof}
We claim that the following identity holds:
$$f_i(\calp[n])+\sum_{k=i+1}^{n}f_k(\calp[n])\nu_{i,k}(\calp)=\binom{n+1}{i+1}.$$
This
 follows at once from Proposition \ref{Prop-f(s(P))} and Corollary \ref{Cor-nuP=nuP[k]}, since $\nu_{i,k}(\calp) = \nu_{i,k}(\calp[n])$ and
since $s(\calp[n])$ is isomorphic to the face poset of an $n$-simplex, and as such has the stated number of $i$-faces. The proposition follows at once, since $f_n(\calp[n]) = 1$.
\end{proof}

\begin{Exmp}
	Let $\calp$ be a finite polyhedral poset, such that for each $x\in\calp$ the sub-poset $\calp_{\le x}$ is isomorphic to the face poset of the $n$-dimensional cube $\calc^n$, for some $n$.
	 Then $\calp$ is regular, and $\nu_{i,j}(\calp)$ can be computed recursively by the formula
	 \[ \nu_{i,j}(\calp) =
	 	{2^r\choose i+1} -\left(f_i(\calc^r) + \sum_{t= \lceil \log_2(i+2)\rceil}^{r-1} 2^{r-t}{r\choose t}\nu_{i,2^t-1}(\calp) \right)\]
if $j=2^r-1$ and $\nu_{i,j}(\calp)=0$ otherwise
\end{Exmp}
\begin{proof}
An $m$-cube has $2^m$ vertices and for each $0\le k\le m$ it has $2^{m-k}{m\choose k}$ $k$-faces. Hence
\[f_i(\calc^m) = \begin{cases}2^{m-k}\binom{m}{k}& \text{if}\; i=2^k-1\\0&\text{otherwise}.\end{cases}.\]
By definition $\nu_{i,k}(\calp) = 0$ unless $k=2^r-1$ for some $r>0$, since there are no objects in the face poset of a cubical complex whose number of vertices is not a power of 2.

Set $n=2^r-1$ in Proposition \ref{Prop-nu_ik}:
\begin{align*}
\nu_{i,2^r-1}(\calp) = &{2^r\choose i+1} -\left(f_i(\calc^r) + \sum_{k=i+1}^{2^r-2}f_k(\calc^r)\nu_{i,k}(\calp) \right) =\\
&{2^r\choose i+1} -\left(f_i(\calc^r) + \sum_{t= \lceil \log_2(i+2)\rceil}^{r-1} f_{2^t-1}(\calc^r)\nu_{i,2^t-1}(\calp) \right)=\\
&{2^r\choose i+1} -\left(f_i(\calc^r) + \sum_{t= \lceil \log_2(i+2)\rceil}^{r-1} 2^{r-t}{r\choose t}\nu_{i,2^t-1}(\calp) \right).
\end{align*}
\end{proof}

Notice that in the example above $\nu_{i,j}(\calp)$ may be nonzero for $i\neq 2^a-1$. For instance for $i=2$ and $r=2$,
$\nu_{2,3}(\calp) = {4\choose 3} - f_2(\calc^2)  = 4$, corresponding to the four 2-faces of $s(\calc^2) = \Delta[3]$, the standard 3-simplex.

For a free graded $\Bbbk$-module $M$, let $P_M(t)$ denote the Poincar\'e series of  $M$. We end with a formula for the Poincar\'e series for the Stanley-Reisner algebra over a finite regular polyhedral poset.

\begin{Prop}
  \label{f-vector}
	Let $\calp$ be a finite regular polyhedral poset and let $\Bbbk[\calp]$ denote its graded Staley-Reisner algebra, where for each $x\in\calp$, the corresponding generator has degree $|x|=|V_\calp(x)|$. Then
  $$P_{\Bbbk[\calp]}(t)=\sum_{i=-1}^\infty\frac{(f_i(\calp)+\sum_{k=i+1}^{|\calp|}f_k(\calp)\nu_{i,k}(\calp))t^{i+1}}{(1-t)^{i+1}}.$$
\end{Prop}

\begin{proof}
By Corollary \ref{Cor-SRposet=SRst-poset}, $\Bbbk[\calp]\cong\Bbbk[s(\calp)]$, and by Lemma \ref{s(P) simplicial}, $s(\calp)$ is a simplicial poset. For any simplicial poset $\cals$, and with the same degree convention on generators
 \begin{equation}\label{SR-Poincare}
 P_{\Bbbk[\mathcal{S}]}(t)=\sum_{i=-1}^\infty\frac{f_i(\mathcal{S})t^{i+1}}{(1-t)^{i+1}},\end{equation}
  by \cite[precise statement]{S1}. By Proposition \ref{Prop-f(s(P))} , $f_i(s(\calp))=f_i(\calp)+\sum_{k=i+1}^{|\calp|}f_k(\calp)\nu_{i,k}(\calp).$ The claim follows by substituting this expression in (\ref{SR-Poincare}).
\end{proof}

\subsection{Acknowledgements}
Daisuke Kishimoto is partly supported by a JSPS KAKENHI grant No. 17K05248.
Ran Levi is partly supported by an EPSRC grant EP/P025072/1. The authors are grateful to the University of Kyoto and the University of Aberdeen for their kind hospitality.

\end{document}